\documentclass[11pt,oneside,reqno]{amsart}
\usepackage{amssymb}
\usepackage{empheq}
\usepackage{fullpage}
\usepackage[colorlinks,allcolors=blue]{hyperref}
\usepackage{mathtools}
\usepackage{aligned-overset}
\usepackage{hyperref}
\usepackage{amsmath}
\numberwithin{equation}{section}
\newtheorem{theorem}{Theorem}[section]

\newtheorem{lemma}[theorem]{Lemma}
\newtheorem{proposition}[theorem]{Proposition}
\theoremstyle{definition}
\newtheorem{definition}[theorem]{Definition}

\newtheorem{question}[theorem]{Question}
\theoremstyle{remark}
\newtheorem{remark}[theorem]{Remark}
\newcommand{\nrm}[1]{\left\Vert#1 \right\Vert}
\newcommand{\abs}[1]{\left\vert #1\right\vert}

\newcommand{\set}[1]{\left\{#1\right\}}

\newcommand{\supp}{\mathrm{supp}}

\newcommand{\bbR}{\mathbb R}

\newcommand{\bbN}{\mathbb N}

\newcommand{\cphi}{C^{\varphi_\alpha}}
\newcommand{\cpsi}{C^{\psi_\alpha}}

\newcommand{\logalpha}[1]{\left( \log \frac{1}{\left \vert #1 \right \vert} \right)^{-\alpha}}
\newcommand{\logonealpha}[1]{\left( \log \frac{1}{\left \vert #1 \right \vert} \right)^{1-\alpha}}

\makeatletter
\@namedef{subjclassname@2020}{\textup{2020} Mathematics Subject Classification}
\makeatother
\title{Conservation and Breakdown of Modulus of Continuity in the Propagation of Loglog Vortex for 2D Euler equation}
\author{Woohyu Jeon}
\address{Department of Mathematical Sciences, Seoul National University, Seoul 08826, Republic of Korea}
\email{woohyu1030@snu.ac.kr}

\subjclass[2020]{35Q31, 35Q35, 76B03}
\keywords{2D Euler equation, loglog vortex, conservation of modulus of continuity, breakdown of modulus of continuity}
\begin{document}
\begin{abstract}
	In this paper, we investigate the persistency of loglog singular structure for 2D Euler equation under the perturbation of continuous function. First, we prove that when loglog vortex is perturbed by the function with modulus of continuity $\mu(r)=\left( \log \frac{1}{r}\right)^{-\alpha} ~~ (0<\alpha<1)$, the loglog vortex is propagated while the perturbation term maintaining same modulus of continuity. We then show that this result is sharp: There exists a smooth function such that when the loglog vortex is perturbed by the smooth function, the norm of perturbation term corresponding to modulus of continuity $\mu(r)=\left( \log \frac{1}{r}\right)^{-\alpha} $ instantly blows up for all $\alpha>1$.
\end{abstract}
\maketitle

\tableofcontents

\section{Introduction}
\label{sec: Introduction}

\subsection{Persistence of singular structure}
\label{subsec: Persistence of singular structure}

Singularity, by which we mean a quantity is infinite, is one of the most important and interesting object in mathematical physics. For example, a fundamental question in fluid dynamics is whether a singularity could be formed from smooth initial data in finite time, e.g., for 3D Navier-Stokes equation and 3D Euler equation. Besides the singularity formation issues, one can alternatively start from singularity and investigate whether the singular structure persists. We illustrate the persistency of singular structure for 2D Euler equation in the vorticity form
\begin{equation} \label{EE}
	\left\{ \begin{aligned}
		& w_t+u \cdot \nabla w=0, \\
		& w(\cdot,0)=w_0,
	\end{aligned} \right.
\end{equation}
where the vorticity $w: \bbR^2 \rightarrow \bbR$ and velocity field $u: \bbR^2 \rightarrow \bbR^2$ are related through Biot-Savart law
\begin{equation} \label{Biot-Savart}
	u(x,t) = \frac{1}{2\pi} \int_{\bbR^2} \frac{(x-y)^\perp}{\abs{x-y}^2}w(y) \, dy
\end{equation}
for $(a,b)^\perp \coloneqq (-b,a)$. The system \eqref{EE}-\eqref{Biot-Savart} is known to be globally well-posded in several standard spaces, e.g., Sobolev space $H^s~~(s>1)$ \cite{MB}, H\"older space $C^{k,\alpha}~~ (0<\alpha<1)$ by Wolibner \cite{Wolibner}, and $L^1 \cap L^\infty$ by Yudovich \cite{Yudovich}\footnote{If a function space is written without domain, its domain is automatically assumed to be $\bbR^2$.}. However, in many physical phenomena, vorticity is sharply concentrated in a small region so that it would be modeled through unbounded or even measure-valued function, which could be regarded as a singularity. One of the most important example of the singularity is so-called point vortex where the vorticity is Dirac measure concentrated in the origin
\begin{equation*}
	w_0(x) = \Gamma \delta_0, \qquad v_0(x) = \frac{1}{2\pi} \frac{x^\perp}{\abs{x}^2}
\end{equation*}
for $\Gamma \in \bbR$. Due to its radial nature, the solution emanating from the point vortex is stationary. However, it is more physically reasonable to consider a point vortex with a perturbation term rather than the exact point vortex. Thus, consider initial vorticity of the following form
\begin{equation} \label{equ13}
	w_0(x) = \Gamma \delta_0+g_0(x)
\end{equation}
for $g_0 \in L^1 \cap L^\infty$. Then, it is natural to ask whether the solution $w(x,t)$ of 2D Euler equation starting from initial vorticity \eqref{equ13} maintains the same structure as \eqref{equ13}, i.e.,
\begin{equation*}
	w(x,t) = \Gamma \delta_{\phi^*(t)}+g(x,t)
\end{equation*}
for continuous map $\phi^*:[0,\infty) \rightarrow \bbR^2$ and $g \in L^\infty_{\textnormal{loc}}([0,\infty); L^1 \cap L^\infty)$. The answer to the question, to the best of our knowledge, was shown to be positive under the additional assumption that $g_0$ is constant near the center of the vortex \cite{Lacave}. Actually, not to mention the solution structure, even the uniqueness for initial data of the form \eqref{equ13} is not known due to the seriousness of Dirac singularity. Therefore, to treat the problem of persistency of singular structure rigorously, it is desirable to consider milder singularity for which global well-posdedness was verified. In fact, there has been several results which guaranteed global well-posedness for specific kinds of unbounded function, e.g., for function with mild $L^p$-norm growth with respect to $p$ by Yudovich \cite{Yudovich2} (see \cite{Crippa-Elementary} for more elementary proof by Crippa--Stefani), for critical Besov space by Vishik \cite{Vishik98}, and for BMO-type space by Bernicot--Keraani \cite{Bernicot14}. Among these admissible class for global well-posedness, the loglog vortex is nearly the most severe singularity \cite{Elgindi2312}, and this is a consequence of Theorem 1.6 \& Theorem 1.8 of \cite{Crippa-Elementary} with its mild $L^p$-norm growth (see Lemma 1 in \cite{Drivas-Propagation})
\begin{equation*}
	\nrm{\log \log \frac{1}{\abs{\cdot}}}_{L^p(B(0,e^{-2}))} \leq M ( \log p+1).
\end{equation*}
 Therefore, for a loglog-type singular vorticity $w_s:\bbR^2 \backslash \set{0} \rightarrow \bbR$ defined by
\begin{equation} \label{def of ws}
	w_s(x) \coloneqq \begin{cases}
		\log \log \frac{1}{\abs{x}}, & (0<\abs{x}<e^{-2}), \\
		s(\abs{x}), & (e^{-2} \leq \abs{x} \leq e^{-1}), \\
		0, & (\abs{x} > e^{-1}),
	\end{cases}
\end{equation}
where $s:(0,\infty) \rightarrow \bbR$ could be any function which makes $w_s$ smooth function with bounded support on $\bbR^2$, we have the following global well-posedness of loglog vortex with a bounded perturbation.

\begin{proposition}
	Let $g$ be any bounded function with compact support. Then, there exists a unique global weak solution $($in the standard distribution sense$)$ for \eqref{EE}-\eqref{Biot-Savart} starting from initial vorticity
	\begin{equation*}
		w_0(x) = w_s +g_0.
	\end{equation*}
	In addition, the solution $w$ is Lagrangian; The solution $w(x,t)$ could be written by
	\begin{equation*}
		w(x,t)=w_0(\Phi_t^{-1}(x)),
	\end{equation*}
	where $\Phi_t^{-1}(x)$ is the inverse flow map of $u~(\coloneqq \nabla^\perp \Delta^{-1} w)$.
\end{proposition}

As we have global well-posedness for initial vorticity of the form
\begin{equation} \label{initial vorticity 1}
	w_0(x) = w_s(x) + \text{(compactly supported bounded function)},
\end{equation}
we can safely return to the following precise question on the persistency of loglog singular structure:
\begin{question} \label{Question 2}
		For a Banach space $X \subset L^\infty$, consider the initial vorticity of the form
		\begin{equation} \label{initial vorticity 2}
			w_0(x)=w_s(x)+g_0(x)
		\end{equation}
		for compactly supported $g_0 \in X$. Then, is it possible to write the solution $w(x,t)$ of 2D Euler equation as the form of
		\begin{equation} \label{solution form vorticity}
			w(x,t)=w_s(x-\phi^*(t))+g(x,t)
		\end{equation}
		for continuous $\phi^*:[0,\infty) \rightarrow \bbR^2$ and $g \in L^\infty_{\textnormal{loc}}([0,\infty); X)$?
\end{question}
	
	To the best of our knowledge, Question \ref{Question 2} was first investigated by Drivas--Elgindi--La \cite{Drivas-Propagation}, where they gave positive answer to for $X=L^\infty$. In this paper, we aim to answer the Question \ref{Question 2} for other spaces $X$ consisting of continuous function. We also provide a sharp regularity threshold which separate the answer to Question \ref{Question 2}.

\bigskip

\subsection{Main Result}
\label{subsec: Main Result}

We first define a function space induced from a modulus of continuity.

\begin{definition} \label{MOC space}
	Let $\mu: [0,c] \rightarrow [0,\infty)$ be a modulus of continuity (MOC), i.e., $\mu$ is a non-decreasing continuous function such that $\mu(0)=0$. Then, the space $C^\mu$ is defined by
	\begin{equation*}
		C^\mu \coloneqq \Big\{ f \in L^\infty \, \Big\vert \, \nrm{f}_{C^\mu} \coloneqq \nrm{f}_{\infty}+ \sup_{0<\abs{x-y}<c} \frac{\abs{f(x)-f(y)}}{\mu(\abs{x-y})} < \infty \Big\}.
	\end{equation*}
\end{definition}

For example, for the modulus of continuities $\mu_1(r)= r \log \frac{1}{r} $ and $\mu_2(r) =r^\alpha$, the space $C^{\mu_1}$ and $C^{\mu_2}$ are the collection of log-Lipschitz function and H\"older continuous function respectively. In this paper, we fix the following three modulus of continuities:
\begin{equation} \label{MOCs}
	\varphi_\alpha(r)=\left( \log \frac{1}{r} \right)^{-\alpha}, \qquad \psi_\alpha(r) = r \left( \log \frac{1}{r}\right)^{1-\alpha}, \qquad \phi_\beta(r) = \left( \log \frac{1}{r} \right)^{-\beta}
\end{equation}
for $\alpha \in (0,1), \beta \in (1,\infty)$.

We are now ready to state the first main result which corresponds to persistence of singular structure in the propagation of loglog vortex $w_s$.

\begin{theorem} [Conservation of modulus of continuity in the propagation of loglog vortex]\label{main thm 1}
	Let $\alpha \in (0,1)$ be given and $g_0 \in C^{\varphi_\alpha}$ be a compactly supported function. Then, the unique solution $w(x,t)$ of 2D Euler equation starting from initial vorticity
	\begin{equation*}
		w_0(x)=w_s(x)+g_0(x)
	\end{equation*}
	could be written as the form of
	\begin{equation*}
		w(x,t)=w_s(x-\phi^*(t))+g(x,t)
	\end{equation*}
	for $g \in L^\infty_{\textnormal{loc}}([0,\infty), \cphi)$. Here, $\phi^*:[0,\infty) \rightarrow \bbR^2$ is the particle trajectory of the origin.
\end{theorem}

Theorem \ref{main thm 1} gives positive answer to Question \ref{Question 2} for $X=\cphi$. While constant vorticity condition is necessary in the neighborhood of point vortex to prove the propagation of point vortex up to bounded correction, Theorem \ref{main thm 1} does not require any specific condition in the neighborhood of loglog vortex. Also, our result shows that loglog vortex is not strong enough to break continuity of perturbation term for wide class of continuous initial vorticity, beyond the boundedness of perturbation term which was shown in \cite{Drivas-Propagation}.

Theorem \ref{main thm 1} leads to the natural question of the existence of other function space $X$ in $\bbR^2$ which also gives positive answer to Question \ref{Question 2}. However, the next main result shows that Theorem \ref{main thm 1} is sharp: The answer to Question \ref{Question 2} is \textit{No} for $X=C^{\phi_\beta}$ for any $\beta \in (1,\infty)$. More precisely:

\begin{theorem} [Breakdown of modulus of continuity in the propagation of loglog vortex] \label{main thm 2}
	There exists a compactly supported smooth function $g_0$ such that for any $\beta\in (1,\infty)$ and $T>0$, the unique solution $w(x,t)$ starting from the initial vorticity
	\begin{equation*}
		w_0(x)=w_s(x)+g_0(x),
	\end{equation*}
	which, by Theorem \ref{main thm 1}, could be written in the form of
	\begin{equation*}
		w(x,t)=w_s(x-\phi^*(t))+g(x,t)
	\end{equation*}
	for $g \in L^\infty([0,T];L^\infty)$ and trajectory of the origin $\phi^*:[0,\infty)\rightarrow \bbR^2$, we have
	\begin{equation} \label{blowup}
		\sup_{t \in [0,T]} \nrm{g(\cdot,t)}_{C^{\phi_\beta}}=\infty.
	\end{equation}
\end{theorem}

Note that the initial perturbation $g_0$ chosen in Theorem \ref{main thm 2} is compactly supported smooth function. Furthermore, as can be seen in its precise definition \eqref{initial data}, $g_0$ is identically zero in a neighborhood of the origin in which the local $C^{\phi_\beta}$-norm blows up instantaneously in the sense of \eqref{blowup}. This implies that the blow-up \eqref{blowup} does not come from a delicate effect of perturbation term, but essentially originates from the structure of loglog singularity. It is also noteworthy that one does not prevent blow-up even though the norm is estimated in lower regularity space whose regularity deteriorates continuously in time, i.e., $H^{s(t)}$ or $C^{0,\alpha(t)}$ (See, e.g., \cite{Chae2308} for this kind of loss of regularity for logarithmically singular surface quasi-geostophic equation).

\begin{remark} \label{Holder remark}
	Since modulus of continuity $\mu(r)=r^\alpha ~~(0<\alpha<1)$ is regular than $\phi_\beta(r)=\left(\log \frac{1}{r}\right)^{-\beta} ~~(\beta>1)$, Theorem \ref{main thm 2} implies breakdown of H\"older continuity for the perturbation term.
\end{remark}

\begin{remark}
	In the definition of singular vortex $w_s$ \eqref{def of ws}, $w_s(x)$ does not need to be equal to $\log \log \frac{1}{\abs{x}}$ exactly near the origin. Instead, any $C^2$ radial function on $\bbR^2 \backslash \set{0}$ which is equivalent to $\log \log \frac{1}{\abs{x}}$ up to second derivative and supported in bounded set suffices to replace $w_s$. More precisely, if a $C^2$ function $q:(0,\infty) \rightarrow \bbR$ satisfies
	\begin{enumerate}
		\item $\supp(q)$ is bounded,
		\item There exist constants $R<e^{-2}, M_1, M_2>0$ such that
		\begin{equation*}
			M_1 \frac{d^k}{dx^k} \left( \log \log \frac{1}{x}\right) \leq \frac{d^k q}{dx^k} \leq M_2 \frac{d^k}{dx^k} \left( \log \log \frac{1}{x} \right)
		\end{equation*}
		for all $0<\abs{x}<R$ and $k=0, 1, 2$,
	\end{enumerate}
	$w_s(x)$ could be replaced with $q(\abs{x})$. This is because, in the proof of our main theorem, we only use information of $w_s$ up to second derivatives.
\end{remark}

\begin{remark}
	If $g_0$ is smooth, something non-trivial could happen only along the center of radial function $w_s$, which was denoted by $\phi^*(t)$. More precisely, by the propagation of regularity for 2D Euler equation (Proposition 8.3 of \cite{MB}), we already know that $w(\cdot,t)$ is smooth at $\bbR^2 \backslash \{\phi^*(t)\}$\footnote{In the Proposition 8.3 of \cite{MB}, there is $L^1 \cap L^\infty$ assumption on $w_0$. However, this $L^1 \cap L^\infty$ condition is only used to get log-Lipschitz flow. For the solution $w$ in \eqref{solution form vorticity}, even though $w_s(x-\phi^*(t))$ does not belong to $L^\infty$, it generates log-Lipschitz velocity field by Lemma \ref{equ7} so that the proof of Proposition 8.3 in \cite{MB} works without modification.}. Therefore, it suffices to focus on value of perturbation term $g$ in \eqref{solution form vorticity} near the center of loglog vortex.
\end{remark}

\begin{remark}
	Note in Theorem \ref{main thm 2} that $g \notin L^\infty([0,T]; C^{\phi_\beta})$ does not imply $g(\cdot, t) \notin C^{\phi_\beta}$ for any fixed $t \in (0,\infty)$. We hope that the blow-up of $C^{\phi_\beta}$-norm for fixed time will be (dis-)proved in the future works.
\end{remark}

\subsection{Notations}
\label{subsec: Notations}

In this subsection, we introduce some notations which will be used throughout this paper.

\begin{enumerate}
	\item For $x \in \bbR^2$ and $r>0$, $B(x,r)$ is a ball of radius $r$ centered at $x$. \vspace{2mm}
	\item For $0<a<b$, the annulus $A(a,b)$ is defined by
	\begin{equation*}
		A(a,b) \coloneqq B(0,b) \backslash B(0,a). \vspace{2mm}
	\end{equation*}
	\item By $C^\infty_c$, we mean the space of compactly supported smooth function.\vspace{2mm}
	\item For each term in \eqref{solution form vorticity}, we denote the velocity field generated by each term by
	\begin{equation} \label{def of velocities}
		u \coloneqq \nabla^\perp \Delta^{-1}w, \qquad u_s \coloneqq \nabla^\perp \Delta^{-1}w_s, \qquad u_g \coloneqq \nabla^\perp \Delta^{-1} g,
	\end{equation}
	so that we could get
	\begin{equation} \label{solution form velocity}
		u(x,t)=u_s(x-\phi^*(t))+u_g(x,t)
	\end{equation}
	from \eqref{solution form vorticity}.\vspace{2mm}
	\item When we write $M=M(\alpha_1, ..., \alpha_n)>0$, we mean a constant $M>0$ depends on variables $\alpha_1, ..., \alpha_n$. Thus, if we write just $M>0$, it is an absolute constant which does not depend on any other variables, or a constant which depends only on fixed solution of 2D Euler throughout this paper. \vspace{2mm}
	\item Although a constant could be changed line by line in successive inequalities, we represent different constants with same symbol for simplicity.
\end{enumerate}

\subsection{Idea of the proof}
\label{subsec: Idea of the proof}

To present the idea for our main results, suppose $g_0$ in \eqref{initial vorticity 2} is symmetric, i.e., $g_0(x)=g(-x)$, just for simplicity. Then, equation for $g$ could be written by
\begin{equation} \label{equ8}
	\left\{ \begin{aligned}
		& g_t+u \cdot \nabla g =-u_g \cdot \nabla w_s, \\
		& g(\cdot,0)=g_0,
	\end{aligned} \right.
\end{equation}
which is a forced transport equation. Here, $u$ is a unique solution of \eqref{EE}-\eqref{Biot-Savart} starting from initial vorticity $w_0$ in \eqref{initial vorticity 2}. By symmetry on $g_0$, we have $u_g(0)=0$ so that it offsets the singularity of $\nabla w_s$, which is equal to $-\frac{1}{\abs{x} \log \frac{1}{\abs{x}}}\frac{x}{\abs{x}}$ near the origin. The following Table \ref{table} shows the regularity of forcing term $-u_g \cdot \nabla w_s$ under the various regularity assumptions on $g$.

\begin{table}[h] 
\begin{tabular}{|c||c|c|c|}
	\hline
	Case & MOC  for $g$ & MOC for $u_g$ & $u_g \cdot \nabla w_s (x)$ \\
	\hline \hline
	(i) & $L^\infty$ & $r \log \frac{1}{r}$ & $\approx 1$ \\
	\hline
	(ii) & $r^\alpha ~~(0<\alpha<1)$ & $r$ & $\approx \left( \log \frac{1}{\abs{x}}\right)^{-1}$ \\
	\hline
	(iii) & $\left( \log \frac{1}{r}\right)^{-\alpha} ~~ (0<\alpha<1)$ & $r \left( \log \frac{1}{r} \right)^{1-\alpha}$ & $\approx \left( \log \frac{1}{\abs{x}}\right)^{-\alpha}$ \\
	\hline
	(vi) & $\left( \log \frac{1}{r} \right)^{-1}$ & $r \log \log \frac{1}{r}$ & $\approx \frac{\log \log \frac{1}{\abs{x}}}{\log \frac{1}{\abs{x}}}$ \\
	\hline
	(v) & $\left( \log \frac{1}{r} \right)^{-\beta} ~~ (\beta>1)$ & $r$ & $\approx \left( \log \frac{1}{\abs{x}}\right)^{-1} $ \\
	\hline
\end{tabular} \caption{Magnitude of forcing term under various modulus of continuity for $g$} \label{table}
\end{table}

In the case (i) and (iii), we can see that assumed modulus of continuity on $g$ (second column) and forcing term on the equation \eqref{equ8} (fourth column) is consistent. This is the fundamental reason the Theorem \ref{main thm 1} holds. Having noticed this consistency, Theorem \ref{main thm 1} could be proven by considering the following standard iteration scheme: Set $g^{(0)}(x,t)=g_0(x)$ and define $g^{(n+1)}$ by the solution of
\begin{equation} \label{equ11}
	\left\{ \begin{aligned}
		& g^{(n+1)}_t+ u \cdot \nabla g^{n+1} = -u_{g^{(n)}} \cdot \nabla w_s, \\
		& g^{(n+1)}(\cdot,0) = g_0.
	\end{aligned} \right.
\end{equation}
Then, it will be shown that the sequence $\{ g^{(n)} \}_{n \in \bbN}$ is the Cauchy sequence in $L^\infty([0, T]; \cphi)$ for any $T \geq 0$ and $\alpha \in (0,1)$. Even if symmetric assumption on $g_0$ is eliminated, the same proof works since offset of singularity in $\nabla w_s(x-\phi^*(t))$ still occurs as can be seen in the general equation for $g$ \eqref{equation for g original}.

In contrast, in the case (ii) and (v), we can see that assumed modulus of continuity on $g$ and forcing term on the \eqref{equ8} is inconsistent: Forcing term is far greater. This is the essential reason the Theorem \ref{main thm 2} and Remark \ref{Holder remark} holds. We briefly sketch the proof of Theorem \ref{main thm 2}. To show $\lim_{t \searrow 0} \nrm{g(\cdot,t)}_{C^{\phi_\beta}}=\infty$, we first construct symmetric initial perturbation $g_0 \in C^\infty_c$ as in \ref{initial data} so that $u_{g_0}(\coloneqq \nabla^\perp \Delta^{-1} g_0)$ is a hyperbolic Lipschitz velocity field satisfying \eqref{Lip initial} and then investigate the value of $g$ along the trajectory $\phi_r(t)$, where $\phi_r: [0,\infty) \rightarrow \bbR^2$ is the unique solution of ODE
\begin{equation}
	\left\{ \begin{aligned}
		& \frac{d}{dt} \phi_r(t) = u(\phi_r(t),t), \\
		& \phi_r(0)=(0,r) \quad (r>0).
	\end{aligned} \right.
\end{equation}
From \eqref{equ8} and the fact that $g_0$ vanishes near the origin, we have
\begin{equation} \label{equ9}
	g(\phi_r(t),t) = -\int_0^t u_g(\phi_r(s),s) \cdot \nabla w_s (\phi_r(s)) \, ds
\end{equation}
for all small $r>0$. Note that when $s=0$, the integrand becomes
\begin{equation*}
	u_g(\phi_r(0),0) \cdot \nabla w_s (\phi_r(0)) \underset{\eqref{def of ws} }{=} u_{g_0}(0,r) \cdot \frac{1}{\abs{(0,r)}\log \frac{1}{\abs{(0,r)}}} \frac{-(0,r)}{\abs{(0,r)}} \underset{\eqref{Lip initial} }{\geq} \frac{K}{\boldsymbol{\log \frac{1}{r}}},
\end{equation*}
which strongly suggests that the forcing term is strong enough to break modulus of continuity less severe than $\left( \log \frac{1}{r}\right)^{-1}$. Thus, to prove Theorem \ref{main thm 2}, it suffices to show the integrand does not deviate far from the value when $s=0$ up to time $t$, i.e., 
\begin{align}
u_g(\phi_r(s),s) \cdot \nabla w_s (\phi_r(s)) &\approx u_g(\phi_r(0),0) \cdot \nabla w_s (\phi_r(0)), \label{equ10} \\
	\abs{\phi_r(s)}& \approx \abs{\phi_r(0)} \label{equ10-1}
\end{align}
for all small $r>0$ and $0 \leq s \leq t$ \footnote{Here, $A \approx B$ means there exists $c>1$ such that $\frac{1}{c}A \leq B \leq cA$}.

However, The main obstacle is that we do not have Lipschitz regularity on velocity field $u$ a priori, even though information on $\nabla u$ is necessary to establish \eqref{equ10}. Furthermore, this lack of regularity could create infinitely small scale in the origin so that the particle trajectory starting at $(0,r)$ could get closer to the origin, a center of loglog singularity near which estimates become worse, far faster than exponential rate (i.e., $\abs{\phi_r(t)}\ll re^{-t}$). This made \eqref{equ10-1} seem infeasible. To overcome this problems, in the earlier version of this paper, which proved breakdown of H\"older continuity mentioned in Remark \ref{Holder remark}, used proof by contradiction; the author made contradictory assumption that perturbation term $g(x,t)$ has uniform in time $C^{0,\alpha}$-norm by which they could avoid aforementioned issues.

In this paper, without adopting proof by contradiction, we use a posteriori estimates available from Section \ref{sec: Conservation of Modulus of Continuity}. Although results in Section \ref{sec: Conservation of Modulus of Continuity} still does not guarantee that $u$ is Lipschitz (precisely, the modulus of continuity for $u$ will be $\mu(r)=r \left( \log \frac{1}{r}\right)^{1-\alpha}$), we show that this regularity is enough to control $\abs{\phi_r(t)}$ for \textit{short time}. To be more precise, we show that for any $\alpha \in (0,1)$, there exist $M>0$, $R=R(\alpha)>0$ and $T>0$ such that for all $0<r<R$ and $0 \leq t \leq T$,
\begin{equation} \label{equ10-2}
	re^{-Mt\left( \log \frac{1}{r}\right)^{1-\alpha}}<\abs{\phi_r(t)} < re^{Mt\left( \log \frac{1}{r}\right)^{1-\alpha}},
\end{equation}
	so that if we take time $t=t(r)$ sufficiently small than $\left( \log \frac{1}{r}\right)^{1-\alpha}$, we could get \eqref{equ10-1}. Also, combining \eqref{equ10-2} with interior Schauder estimate, we could get estimate for $\abs{\nabla u(x)}$ for all $x\in \bbR^2 \backslash \set{0}$ (Proposition \ref{del u proposition}), and this leads to \eqref{equ10-1} without a global Lipschitz estimate on $u$. Using \eqref{equ10}-\eqref{equ10-1}, we get Theorem \ref{main thm 2} almost as a corollary.

\subsection{Organization of the paper}
\label{subsec: Organization of the paper}

We prove the conservation of modulus of continuity (Theorem \ref{main thm 1}) in Section \ref{sec: Conservation of Modulus of Continuity} and breakdown of modulus of continuity (Theorem \ref{main thm 2}) in Section \ref{sec: Breakdown of Modulus of Continuity}.

In Section \ref{subsec: Elementary Lemmas}, we gather elementary lemmas to apply iteration scheme \eqref{equ11}. In Section \ref{subsec: Proof of the Main Theorem: Conservation Part}, we prove Theorem \ref{main thm 1} by completing the iteration scheme.

In Section \ref{subsec: Construction of Initial Data and Basic Properties of the Solution}, we construct $g_0$ in \eqref{initial vorticity 2} and get some estimates for the solution $(w,u)$ starting from initial vorticity $w_0$ in \eqref{initial vorticity 2}. In Section \ref{subsec: The Key Lemma}, we gather results in Section \ref{subsec: Construction of Initial Data and Basic Properties of the Solution} to prove the Key Lemma (Lemma \ref{key lemma}). Then, we finish the proof of Theorem \ref{main thm 2} in Section \ref{subsec: Proof of the Main Theorem: Breakdown Part} as a corollary of Proposition \ref{key lemma}. 

\bigskip

\textbf{Acknowledgements.} The author thanks In-Jee Jeong for fruitful discussions. The author also acknowledges partial support by Samsung Science and Technology Foundation (SSTF-BA2002-04).

\section{Conservation of Modulus of Continuity}
\label{sec: Conservation of Modulus of Continuity}

In this section, we prove Theorem \ref{main thm 1}. For simplicity, we prove the following theorem where we assume the center of loglog vortex is fixed at the origin due to symmetric property of $g_0$ in initial data \eqref{initial vorticity 2}
\begin{theorem} \label{main thm 1 sub}
	Let $\alpha \in (0,1)$ be given. Suppose compactly supported $g_0 \in C^{\varphi_\alpha}$ is symmetric with respect to the origin $($i.e., $g_0(-x)=g_0(x)$ for all $x \in \bbR^2$$)$ so that perturbation term $g(x,t)$ in \eqref{solution form vorticity} is symmetric with respect to the origin and $u_g$ in \eqref{solution form velocity} is odd-symmetric with respect to the origin $($i.e., $u_g(x,t)=-u_g(x,t)$ for all $x\in \bbR^2, t \in [0,\infty)$$)$. Then, $ g \in L^\infty_{\textnormal{loc}}([0,\infty); C^{\varphi_\alpha})$.
\end{theorem} 

To prove Theorem \ref{main thm 1 sub}, we gather a few elementary lemmas.

\subsection{Elementary Lemmas}
\label{subsec: Elementary Lemmas}

\begin{lemma} \label{vor to vel lemma}
	Let $\alpha \in (0,1)$ be given. Then, there exists a constant $C=C(\alpha, \abs{\supp (w)})$ such that for compactly supported $w \in \cphi$, we have $u=\nabla^\perp \Delta^{-1}w \in \cpsi$ with
	\begin{equation} \label{vor to vel}
		\nrm{u}_{\cpsi} \leq C \nrm{w}_{\cphi}.
	\end{equation}
\end{lemma}

\begin{proof}
	The proof is almost same as the proof of the fact that velocity field generated by $L^1 \cap L^\infty$ vorticity is log-Lipschitz (see Lemma 8.1 of \cite{MB}). From the proof of Lemma 8.1 of \cite{MB}, there exists $C>0$ such that for $\abs{x-y}<1$,
	\begin{equation} \label{eq1}
		\begin{aligned}
			& \abs{u(x)-u(y)} \\
			\underset{\substack{\text{Lem 8.1 } \text{in \cite{MB}}} }&{\leq} C ( \nrm{w}_1 + \nrm{w}_{\infty}) \abs{x-y}+ C\abs{x-y} \int_{B(x,1) \backslash B(x, 2\abs{x-y})} \frac{\abs{w(x)-w(z)}}{\abs{x-z}^2} \, dz \\
			& \leq C ( \nrm{w}_1 + \nrm{w}_{\infty}) \abs{x-y}+ C\nrm{w}_{\cphi}\abs{x-y}  \int_{B(x,1) \backslash B(x, 2\abs{x-y})} \frac{\logalpha{x-z}}{\abs{x-z}^2} \, dz \\
			&= C ( \nrm{w}_1 + \nrm{w}_{\infty}) \abs{x-y}+ \frac{C}{1-\alpha}  \nrm{w}_{\cphi} \abs{x-y} \left( \log \frac{1}{2\abs{x-y}}\right)^{1-\alpha} \\
			&\leq C \nrm{w}_{\cphi} \abs{x-y}\left( \log \frac{1}{\abs{x-y}}\right)^{1-\alpha},
		\end{aligned}
	\end{equation}
	where $C$ is a constant depending only on $\alpha$ and $\abs{\supp (w)}$. Combining \eqref{eq1} with
	\begin{equation*}
		\nrm{u}_{\infty} \leq \nrm{w}_{1} + \nrm{w}_{\infty} \leq (\abs{\supp (w)}+1) \nrm{w}_{\cphi}
	\end{equation*}
	yields the inequality \eqref{vor to vel}.
\end{proof}

Next, we point out the fact that under the log-Lipschitz vector field, the two points could get closer or farther at most double exponential rate.

\begin{lemma}
	Let $u$ be a log-Lipschitz vector field such that
	\begin{equation*}
		\frac{u(x,t)-u(y,t)}{\abs{x-y} \log \frac{1}{\abs{x-y}}} \leq N
	\end{equation*}
	for all $\abs{x-y} < e^{-1}, 0 \leq t \leq T$, and $\Phi(\cdot,t)$ be its flow map. Then, for $\abs{x-y}<e^{-NT}$,
	\begin{equation} \label{access rate}
		\abs{x-y}^{e^{Nt}} \leq \abs{\Phi(x,t)-\Phi(y,t)} \leq \abs{x-y}^{e^{-Nt}}.
	\end{equation}
\end{lemma}

\begin{proof}
	This follows directly from the comparison with ODE
	\begin{equation*}
		\left\{ \begin{aligned}
			x'(t)&= \pm Nx(t)\log \frac{1}{x(t)}, & (0 \leq t \leq T), \\
			x(0)&=x_0 ,& (x_0 < e^{-NT}),
 		\end{aligned} \right.
	\end{equation*}
	whose solution is $x(t)= x_0^{e^{\mp Nt}}$.
\end{proof}

To prove Theorem \ref{main thm 1 sub}, we now derive the equation for perturbation term $g$ in \eqref{solution form vorticity}. Since $(w,u)$ is a vorticity-velocity pair for 2D Euler equation, solution form \eqref{solution form vorticity}, \eqref{solution form velocity} give
\begin{equation*}
	(\underbrace{w_s(x-\phi^*(t))+g}_{=w})_t + ( \underbrace{u_s(x-\phi^*(t))+u_g}_{=u}) \cdot \nabla (w_s(x-\phi^*(t))+g) = 0.
\end{equation*}
Then, we have
\begin{equation} \label{eq2}
	g_t+u \cdot \nabla g = - \left( u_g - \frac{d \phi^*}{dt}(t) \right) \cdot \nabla w_s(x-\phi^*(t)) -u_s(x-\phi^*(t)) \cdot \nabla w_s(x-\phi^*(t)).
\end{equation}
Since $w_s$ is radial and $w_s(x)= \log \log \frac{1}{\abs{x}}$ for $0<\abs{x}<e^{-2}$, the explicit formula $u_s(x)= \left( \frac{1}{\abs{x}}\int_0^{\abs{x}}r \log \log \frac{1}{r}\, dr \right)\mathbf{e}_\theta$ gives
\begin{equation*}
	\lim_{x \rightarrow 0} \frac{\abs{u_s(x)}}{\abs{x}  \log \log \frac{1}{\abs{x}}}=\frac{1}{2},
\end{equation*}
which yields $u_s(0)=0$. Therefore,
\begin{equation}\label{eq3}
	\frac{d\phi^*}{dt}(t) = u(\phi^*(t),t) \underset{\eqref{solution form velocity}}{=} \underbrace{u_s(\phi^*(t)-\phi^*(t)) }_{=0}+u_g(\phi^*(t),t)=u_g(\phi^*(t),t).
\end{equation}
Since radial symmetry of $w_s$ makes $u_s$ and $\nabla w_s$ orthogonal, \eqref{eq2} could be simplified to
\begin{equation} \label{equation for g original}
	g_t+u \cdot \nabla g = -(u_g(x,t) -u_g(\phi^*(t),t)) \cdot \nabla w_s (x-\phi^*(t)).
\end{equation}
Here, just for simplicity, suppose initial perturbation term $g_0$ is symmetric with respect to the origin. Then, initial velocity field $u_0$ satisfies $u_0(-x)=-u_0(x)$, so does $u$ by a symmetric property of the Euler equation. As $u(0,t)=0$ for all $t \in [0,\infty)$, $\phi^*(t) \equiv 0$ for all $t \in [0,\infty)$ so that \eqref{eq2} could be simplified to
\begin{equation} \label{equation for g}
	g_t+u \cdot \nabla g = -u_g \cdot \nabla w_s.
\end{equation}
This is nothing but a transport equation with a forcing term. To prove \ref{main thm 1 sub} using \eqref{equation for g}, one have to a priori expect that for compactly supported symmetric $g\in \cphi$, forcing term on the right hand side also belongs to space $\cphi$. Heuristically, this is plausible as $\abs{u_g(x)} \approx \abs{x} \logonealpha{x}$ by \eqref{vor to vel} and $\abs{\nabla w_s(x)} \approx \frac{1}{\abs{x} \log \frac{1}{\abs{x}}}$ near the origin\footnote{Even though we remove the assumption of symmetry on $g_0$, this cancellation of `zero times infinity' still occurs near $\phi^*(t)$ on the right hand side of \eqref{equation for g original}. This guarantees that similar proof goes on without assumption on symmetry.}. More precisely, we prove
\begin{lemma} \label{forcing term}
	Let $\alpha \in (0,1)$ be given. Suppose a vector field $v \in \cpsi$ satisfies $v(0)=0$. Then, there exists a constant $C>0$ such that the scalar function $v \cdot \nabla w_s \in \cphi$ with
	\begin{equation*}
		\nrm{v \cdot \nabla w_s}_{\cphi} \leq C \nrm{v}_{\cpsi}.
	\end{equation*}
\end{lemma}

\begin{proof}
	Suppose two points $x$ and $y$ are given with $\abs{x-y} \leq \frac{e^{-2}}{2}$. Without loss of generality, we assume $\abs{x} \leq \abs{y}$. We divide the case.
	
	\bigskip
	$\bullet$ \underline{Case I}. $\abs{y} \geq e^{-2}$
	
	Note that $\nabla w_s$ is compactly supported smooth function on $\bbR^2 \backslash B \left( 0, \frac{e^{-2}}{2} \right)$ so that
	\begin{equation*}
		\nrm{\nabla w_s}_{\cphi \left(\bbR^2 \backslash B \left( 0, \frac{e^{-2}}{2}\right)\right)} \coloneqq M< \infty.
	\end{equation*}
	 Then, we have
	\begin{align*}
		&\frac{\abs{v(x) \cdot \nabla w_s(x) - v(y) \cdot \nabla w_s(y)}}{\logalpha{x-y}} \\
		&\leq \frac{\abs{v(x)-v(y)} \cdot \abs{\nabla w_s(x)}}{\abs{x-y} \logonealpha{x-y}} \abs{x-y} \log \frac{1}{\abs{x-y}}+\abs{v(y)} \frac{\abs{\nabla w_s(x)-\nabla w_s(y)}}{\logalpha{x-y}} \\
		&\leq \left( \frac{e^{-2}}{2} \log (2e^2 ) + M \right) \nrm{v}_{\cpsi}.
	\end{align*}
	
	\bigskip
	
	$\bullet$ \underline{Case II}. $\abs{y} \leq e^{-2}$
	
	Note that for all $z \in B \left( 0, e^{-2} \right)$,
	\begin{equation} \label{eq4}
		\frac{\abs{v(z) \cdot \nabla w_s(z)}}{\logalpha{z}} \underset{\substack{\eqref{def of ws} }}{\leq} \frac{\nrm{v}_{\cpsi}\abs{z} \logonealpha{z} \cdot \frac{1}{\abs{z} \log \frac{1}{\abs{z}}}}{\logalpha{z}} = \nrm{v}_{\cpsi}.
	\end{equation}
	Now, we divide the Case II into two cases based on size of $\abs{x-y}$ and $\abs{y}$.
	
	\bigskip
	
	$\bullet$ \underline{Case II-(i)}. $\abs{y}\leq e^{-2},~~ \abs{x-y} \geq \frac{\abs{y}}{2}$
	
	In this case, we estimate $\cphi$-norm of $v \cdot \nabla w_s$ as follows,
	\begin{align*}
		\frac{\abs{v(x) \cdot \nabla w_s(x) - v(y) \cdot \nabla w_s(y)}}{\logalpha{x-y}} &\leq \frac{\abs{v(x) \cdot \nabla w_s(x)}}{\logalpha{x-y}}+\frac{\abs{v(y) \cdot \nabla w_s(y)}}{\logalpha{x-y}} \\
		&\leq \frac{\abs{v(x) \cdot \nabla w_s(x)}}{\logalpha{x}} \cdot \frac{\logalpha{x}}{\logalpha{x-y}} + \frac{\abs{v(y) \cdot \nabla w_s(y)}}{\logalpha{y}} \cdot \frac{\logalpha{y}}{\logalpha{x-y}}.
	\end{align*}
	Combining inequality \eqref{eq4} with
	\begin{equation*}
		\frac{\logalpha{x}}{\logalpha{x-y}} \leq 2^\alpha , \qquad \frac{\logalpha{y}}{\logalpha{x-y}} \leq 2^\alpha,
	\end{equation*}
	which is implied by the condition $\abs{x-y}\geq \frac{\abs{y}}{2}$, we have
	\begin{equation*}
		\frac{\abs{v(x) \cdot \nabla w_s(x) - v(y) \cdot \nabla w_s(y)}}{\logalpha{x-y}} \leq 2^{\alpha+1} \nrm{v}_{\cpsi}.
	\end{equation*}
	
	\bigskip
	
	$\bullet$ \underline{Case II-(ii)}. $\abs{y}\leq e^{-2},~~ \abs{x-y} \leq \frac{\abs{y}}{2}$
	
	In this case, we have
	\begin{equation*}
		\frac{\abs{v(x) \cdot \nabla w_s(x) - v(y) \cdot \nabla w_s(y)}}{\logalpha{x-y}} \leq \frac{\abs{v(x)-v(y)}}{\logalpha{x-y}} \cdot \abs{\nabla w_s(x)} + \abs{v(y)} \cdot \frac{\abs{\nabla w_s(x) - \nabla w_s (y)}}{\logalpha{x-y}}.
	\end{equation*}
	Note that $\abs{x-y}\leq \frac{\abs{y}}{2}$ implies $\abs{x-y} \leq \abs{x}$ and $\abs{y} \leq 2\abs{x}$. Then,
	\begin{equation*}
		\frac{\abs{v(x)-v(y)}}{\logalpha{x-y}} \cdot \abs{\nabla w_s(x)} \underset{\eqref{def of ws} }{=} \frac{\abs{v(x)-v(y)}}{\abs{x-y} \logonealpha{x-y}} \cdot \frac{\abs{x-y} \log \frac{1}{\abs{x-y}}}{\abs{x} \log \frac{1}{\abs{x}}}\leq \nrm{v}_{\cpsi}.
	\end{equation*}
	Also,
	\begin{align*}
		\abs{v(y)} \cdot \frac{\abs{\nabla w_s(x) - \nabla w_s (y)}}{\logalpha{x-y}} &\leq \frac{\abs{v(y)}}{\logalpha{x-y}} \cdot \abs{x-y} \frac{2}{\abs{x}^2 \log \frac{1}{\abs{x}}} \\
		&= \frac{2 \abs{v(y)}}{\abs{y} \logonealpha{y}}\cdot \frac{\abs{y} \logonealpha{y}}{\abs{x} \logonealpha{x}} \cdot \frac{\abs{x} \logalpha{x}}{\abs{x-y} \logalpha{x-y}}\\
		&\leq 4 \nrm{v}_{\cpsi}.
	\end{align*}
	As we have calculated $\cphi$-seminorm so far, we now calculate $L^\infty$-norm of $v \cdot \nabla w_s$. For $0 < \abs{z} < e^{-2}$,
	\begin{equation*}
		\abs{v(z) \cdot \nabla w_s (z)} \underset{\eqref{eq4} }{\leq} \nrm{v}_{\cpsi} \sup_{z \in B (0, e^{-2})} \logalpha{z} \leq \nrm{v}_{\cpsi}.
	\end{equation*}
	For $\abs{z} \geq e^{-2}$,
	\begin{equation*}
		\abs{v(z) \cdot \nabla w_s(z)} \leq \nrm{v}_{\cpsi} \sup_{z \in \bbR^2 \backslash B(0,e^{-2})} \abs{\nabla w_s(z)}.
	\end{equation*}
	Combining above inequalities implies the Lemma \ref{forcing term}.
\end{proof}

Gathering above elementary lemmas, we can prove the following proposition. Here, $LL$ is the space of log-Lipschitz function, i.e.,
\begin{equation*}
	LL \coloneqq \set{f \, \Big\vert \, \nrm{f}_{LL} \coloneqq \nrm{f}_{\infty}+\sup_{\abs{x-y}<e^{-1}}\frac{
	\abs{f(x)-f(y)}}{\abs{x-y} \log \frac{1}{\abs{x-y}}}<\infty}.
\end{equation*}

\begin{proposition} \label{consistency lemma}
	Suppose $v \in L^\infty([0,T]; \cpsi)$ with $v(0,t)=0$ for $t\in [0,T]$ and $u \in L^\infty([0,T];LL)$. Then, there exists $C=C(\alpha, T,\nrm{u}_{L^\infty_T LL})>0$ such that the unique weak solution for forced transport equation
	\begin{equation} \label{transport eq}
		\left\{ \begin{aligned}
			& g_t+u \cdot \nabla g = - v \cdot \nabla w_s, \\
			& g(\cdot,0)=g_0
		\end{aligned} \right.
	\end{equation}
	belongs to $L^\infty([0,T]; \cphi)$ with
	\begin{equation} \label{consistency}
		\nrm{g(\cdot,t)}_{\cphi} \leq C(\nrm{g_0}_{\cphi} + t \nrm{v}_{L^\infty_t\cpsi}).
	\end{equation}
\end{proposition}

\begin{proof}
	Let $\nrm{u}_{L^\infty_T LL} \eqqcolon N$ and take two points $X, Y$ with $\abs{X-Y}\leq e^{-NT}$. Also, fix a time $t \in [0,T]$. Then, there exists $x, y \in \bbR^2$ such that $\Phi(x,t)=X, \Phi(y,t)=Y$ with $\abs{x-y} \leq e^{-1}$, where $\Phi: \bbR^2 \times [0,T] \rightarrow \bbR^2$ is a flow map generated by $u$. Since the solution for \eqref{transport eq} could be explicitly written by
	\begin{equation} \label{eq5}
		g(\Phi(z,t),t) = g_0(z)-\int_0^t v(\Phi(z,s),s) \cdot \nabla w_s (\Phi(z,s),s) \, ds,
	\end{equation}
	we have
	\begin{equation} \label{eq6}
		\begin{aligned}
			&\frac{\abs{g(\Phi(x,t),t)-g(\Phi(y,t),t)}}{\logalpha{\Phi(x,t)-\Phi(y,t)}} \leq \frac{\abs{g_0(x)-g_0(y)}}{\logalpha{x-y}} \cdot \frac{\logalpha{x-y}}{\logalpha{\Phi(x,t)-\Phi(y,t)}}\\ &+ \int_0^t \underbrace{\frac{\abs{v(\Phi(x,s),s) \cdot \nabla w_s (\Phi(x,s),s)-v(\Phi(y,s),s) \cdot \nabla w_s (\Phi(y,s),s)}}{\logalpha{\Phi(x,s)-\Phi(y,s)}}}_{
			\underset{\text{Lem }\ref{forcing term}}{\leq} C \sup_{0 \leq s \leq t} \nrm{v(\cdot,s)}_{\cpsi}} \cdot \frac{\logalpha{\Phi(x,s)-\Phi(y,s)}}{\logalpha{\Phi(x,t)-\Phi(y,t)}} \, ds 	
		\end{aligned}
	\end{equation}
	As we have 
	\begin{equation*}
		\frac{\logalpha{x-y}}{\logalpha{\Phi(x,t)-\Phi(y,t)}} \leq e^{N \alpha t}, \qquad \frac{\logalpha{\Phi(x,s)-\Phi(y,s)}}{\logalpha{\Phi(x,t)-\Phi(y,t)}} \leq e^{2N\alpha t},
	\end{equation*}
	by \eqref{access rate}, \eqref{eq6} leads to
	\begin{equation} \label{eq7}
		\sup_{\abs{X-Y}\leq e^{-NT}}\frac{\abs{g(X,t)-g(Y,t)}}{\logalpha{X-Y}} \leq e^{N \alpha t} \nrm{g_0}_{\cphi} + Cte^{2N\alpha t}\nrm{v}_{L^\infty_t \cpsi},
	\end{equation}
	 where $C$ is the constant in Lemma \ref{forcing term}.
	 
	 Similarly, from \eqref{eq5},
	 \begin{equation} \label{eq8}
	 	\nrm{g(\cdot,t)}_{\infty} \leq \nrm{g_0}_{\infty} + t \nrm{v \cdot \nabla w_s}_{L^\infty_t L^\infty} \underset{\text{Lem }\ref{forcing term} }{\leq} \nrm{g_0}_{\cphi}+Ct \nrm{v}_{L^\infty_t \cpsi}.
	 \end{equation}
	 Combining \eqref{eq7} and \eqref{eq8} implies Proposition \ref{consistency lemma}.
\end{proof}

\subsection{Proof of the Main Theorem: Conservation Part}
\label{subsec: Proof of the Main Theorem: Conservation Part}

In this subsection, we prove Theorem \ref{main thm 1 sub} and Theorem \ref{main thm 1} using the results in Section \ref{subsec: Elementary Lemmas}. The proof is quite standard.

\bigskip

\underline{\textit{Proof of Theorem \ref{main thm 1 sub}}}.

Let $T<\infty$ be given. Define a sequence $\set{g^{(n)}}_{n \in \bbN}$ inductively by setting $g^{(0)}(\cdot,t) \equiv g_0$ and $g^{(n+1)}$ being a solution of forced transport equation
\begin{equation} \label{iteration scheme}
	\left\{ \begin{aligned}
		& g^{(n+1)}_t + u \cdot \nabla g^{(n+1)}=-u_{g^{(n)}} \cdot \nabla w_s, \\
		& g^{(n+1)}(x,0)=g_0(x),
	\end{aligned} \right.
\end{equation}
where $u \in L^\infty ([0,T]; LL)$ is the fixed velocity field which is the solution of 2D Euler equation
\begin{equation*}
	\left\{ \begin{aligned}
		& w_t+ u \cdot \nabla w=0, \\
		& w(\cdot,0)=w_s + g_0.
	\end{aligned} \right.
\end{equation*}
Note that if $g^{(n)}\in \cphi$ is symmetric with respect to the origin for all $t \in [0,T]$, then $g^{(n+1)}$ is well-defined through Lemma \ref{vor to vel lemma} and Proposition \ref{consistency lemma}, and symmetric with respect to the origin. Therefore,  $\set{g^{(n)}}_{n \in \bbN}$ is well-defined for all $n \in \bbN$. Since $g^{(n+1)} - g^{(n)}$ satisfies the forced transport equation
\begin{equation*}
	\left\{ \begin{aligned}
		& (g^{(n+1)} - g^{(n)})_t + u \cdot \nabla (g^{(n+1)} - g^{(n)}) = -(u_{g^{(n+1)}}-u_{g^{(n)}}) \cdot \nabla w_s, \\
		& (g^{(n+1)} - g^{(n)})(x,0) =0.
	\end{aligned} \right.
\end{equation*}
Proposition \ref{consistency lemma} and Lemma \ref{vor to vel lemma} implies that there exists a constant $C=C(\alpha, T, \nrm{u}_{L^\infty_T LL}, \abs{\supp g_0})$ such that
\begin{equation} \label{eq9}
	\nrm{g^{(n+1)} - g^{(n)}}_{L^\infty_t \cphi} \underset{\text{Lem } \ref{consistency lemma} }{\leq}Ct \nrm{u_{g^{(n)}}-u_{g^{(n-1)}}}_{L^\infty_t \cpsi} \underset{\eqref{vor to vel} }{\leq} Ct \nrm{g^{(n+1)} - g^{(n)}}_{L^\infty_t \cphi}.
\end{equation}
If we take $T_0>0$ sufficiently small so that $CT_0 \leq \frac{1}{2}$, \eqref{eq9} implies that  $\set{g^{(n)}}_{n \in \bbN}$ is a Cauchy sequence in $L^\infty_{T_0}\cphi$. Therefore, there exists $h \in L^\infty_{T_0}\cphi$ so that $g^{(n)} \rightarrow h$ in $L^\infty_{T_0}\cphi$. Since $u \in L^\infty_T L^\infty$ and $w_s$ has compact support regardless of $n \in \bbN$, we can iterate this scheme so that maximal convergence time could be extended up to $T$ like in the proof of classical Cauchy-Lipschitz theorem.

Also, since
\begin{equation*}
	g^{(n)} \rightarrow h \qquad \text{and} \qquad u_{g^{(n)}} \cdot \nabla w_s \xrightarrow[\text{Lem }\ref{forcing term} ]{} u_h \cdot \nabla w_s \qquad \text{in }~~ L^\infty_T L^\infty, 
\end{equation*}
it is straightforward to show that $h$ is the weak solution of \eqref{equation for g} by taking $n \rightarrow \infty$ in the weak formulation of \eqref{iteration scheme}
\begin{equation*}
	\int_{\bbR^2} \left( \phi(x,T) g^{(n+1)}(x, T) - \phi(x,0) g_0 \right)\, dx + \int_0^T \int_{\bbR^2} \left( -\frac{D \phi}{Dt}g^{(n+1)} + (u_{g^{(n)}}\cdot \nabla w_s)\phi \right)\, dxdt=0
\end{equation*}
for a test function $\phi \in C^\infty_c (\bbR^2 \times [0,T])$.

Lastly, we conclude our proof by showing that the unique solution $g$ of \eqref{equation for g} whose regularity is guaranteed to be at most $L^\infty_TL^\infty$ at this point, is equal to $h \in L^\infty_T\cphi$. Since both $g$ and $h$ are weak solutions of \eqref{equation for g}, $g-h$ satisfies the following transport equation
\begin{equation*}
	\left\{ \begin{aligned}
		&(g-h)_t + u \cdot \nabla (g-h) = - (u_g-u_h) \cdot \nabla w_s, \\
		& (g-h)(\cdot,0)=0.
	\end{aligned} \right.
\end{equation*}
Thus, we have $L^\infty$- estimate
\begin{equation} \label{eq10}
	\nrm{g-h}_{L^\infty_t L^\infty} \leq \int_0^t \nrm{(u_g-u_h)(\cdot,s) \cdot \nabla w_s(\cdot,s)}_{L^\infty} \, ds.
\end{equation}
Since $u_g$ and $u_h$ are log-Lipschitz and $u_g(0)=u_h(0)=0$ by symmetry, we have for $z \in B(0,e^{-2})$,
\begin{equation*}
	\abs{(u_g-u_h)(z) \cdot \nabla w_s(z))} \underset{\eqref{def of ws}}{\leq} C \nrm{g-h}_\infty \abs{z} \log \frac{1}{\abs{z}} \cdot \frac{1}{\abs{z} \log \frac{1}{\abs{z}}} = C \nrm{g-h}_{\infty},
\end{equation*}
where $C$ is a constant depending only on $\abs{\supp (g-h)}$ which is already known to be bounded. Also, for $z \in \bbR^2 \backslash B(0,e^{-2})$,
\begin{equation*}
	\abs{(u_g-u_h)(z) \cdot \nabla w_s(z)} \leq C \nrm{g-h}_\infty \sup_{z \in \bbR^2 \backslash B(0,e^{-2})} \abs{\nabla w_s(z)}.
\end{equation*}
Therefore, there exists a absolute constant $C>0$ such that
\begin{equation} \label{eq11}
	\nrm{(u_g-u_h)\cdot \nabla w_s}_\infty \leq C \nrm{g-h}_{\infty}.
\end{equation}
Combining \eqref{eq10}, \eqref{eq11} with Gr\"onwall inequality leads to $g=h$ on $[0,T]$. \hfill $\square$

\bigskip

\underline{\textit{Proof of Theorem \ref{main thm 1}}}.

In the proof of Theorem \ref{main thm 1 sub}, the reason we impose symmetry condition on $g_0$ is to apply Lemma \ref{forcing term} and Proposition \ref{consistency lemma}. The vanishing velocity condition in Lemma \ref{forcing term} is used to cancel singularity of $\nabla w_s$ at the origin. However, investigating general equation for $g$ \eqref{equation for g original}, where there is no symmetry condition on $g$, it is immediate to see that vanishing velocity condition is redundancy since $u_g(x,t) - u_g(\phi^*(t),t)$ automatically has zero velocity at the singular point of $\nabla w_s(x-\phi^*(t))$. For this reason, similar proof goes on with obvious modification to Theorem \ref{main thm 1 sub} without any symmetry condition.$\hfill \square$

\section{Breakdown of Modulus of Continuity}
\label{sec: Breakdown of Modulus of Continuity}

In this section, we prove Theorem \ref{main thm 2}. The first thing we have to do is to clarify initial perturbation $g_0$.

\subsection{Construction of Initial Data and Basic Properties of the Solution}
\label{subsec: Construction of Initial Data and Basic Properties of the Solution}

Let $\eta \in C^\infty_c$ be a standard mollifier and $\eta_\varepsilon \coloneqq \frac{1}{\varepsilon^2}\eta \left( \frac{\cdot}{\varepsilon} \right)$. For a function $c_0$ defined by
\begin{equation*}
	c_0(x)= \begin{cases}
			1, & (x\in [21,23] \times [21,23]  ~~ \cup ~~ [-23, -21] \times [-23, -21]), \\
			-1, & (x\in [-23,-21] \times [21,23]  ~~ \cup ~~ [21,23] \times [-23,-21]), \\
			0, & (\text{otherwise}),
		\end{cases}
\end{equation*}
which is symmetric with respect to the origin, define $g_0$ by
\begin{equation} \label{initial data}
	g_0 \coloneqq \eta_\varepsilon * c_0.
\end{equation}
Then, an elementary calculation shows that we can choose small $\varepsilon>0$ such that there exist $K, \delta >0$ so that $u_{g_0}^2$, the second component of $u_{g_0}$, satisfies
\begin{equation} \label{Lip initial}
	-(K+1)r\leq u_{g_0}^2(0,r) \leq -Kr
\end{equation}
for all $0<r<\delta$.

Let $w$ be the unique weak solution of 2D Euler equation
\begin{equation} \label{equ1}
	\left\{ \begin{aligned}
		& w_t+u \cdot \nabla w=0, \\
		& w(\cdot,0)=w_s(x) +g_0.
	\end{aligned} \right.
\end{equation}
Then, by Theorem 3 of \cite{Drivas-Propagation}, the solution $w(x,t)$ could be written as the form of
\begin{equation*}
	w(x,t)=w_s(x-\phi^*(t))+g(x,t), 
\end{equation*}
where $\phi^*(t)$ is the particle trajectory of the origin and $g$  belongs to $ L^\infty_{\textnormal{loc}} ([0,\infty), L^\infty)$. From now on, by $w, u$ and $g$, \textbf{we always mean above fixed solution, not a general solution}.

\begin{lemma} \label{symmetry lemma}
	The velocity field $u(x,t)$ satisfies $u(-x,t)=-u(x,t)$ for all $x \in \bbR^2, t \in [0,\infty)$ so that $u(0,t)=0$ for all $t \in [0,\infty)$. Therefore, $\phi^*(t) \equiv 0$ so that the solution can be written in the form of
	\begin{equation*}
		w(x,t)=w_s(x) + g(x,t).
	\end{equation*} 
\end{lemma}
\begin{proof}
	Clear as both $w_s(x)$ and $g_0$ are symmetric with respect to the origin.
\end{proof}

Now, we derive an estimate for velocity field $u_s, u_g$ and $u$ defined as in \eqref{def of velocities}.

\begin{lemma} \label{equ7}
	There exists $R>0$ such that for all $0<\abs{x}<R$,
	\begin{equation} \label{us}
		\abs{u_s(x)} \leq \abs{x} \log \log \frac{1}{\abs{x}},
	\end{equation}
	\begin{equation} \label{del us}
		\abs{\nabla u_s(x)} \leq 3 \log \log \frac{1}{\abs{x}}.
	\end{equation}
	Also,
	\begin{equation*}
		\nrm{u_s}_{\infty} + \sup_{\abs{x-y}<e^{-3}}\frac{\abs{u_s(x)-u_s(y)}}{\abs{x-y}\log \log \frac{1}{\abs{x-y}}} < \infty.
	\end{equation*}
\end{lemma}

\begin{proof}
	Let $\mathbf{e}_r= \begin{pmatrix}
		\frac{x_1}{\abs{x}} \\
		\frac{x_2}{\abs{x}}
	\end{pmatrix}, \mathbf{e}_\theta=\begin{pmatrix}
		-\frac{x_2}{\abs{x}} \\
		\frac{x_1}{\abs{x}}
	\end{pmatrix}$. Remembering the definition of $w_s$ \eqref{def of ws}, direct calculation (e.g., see Example 2.1 of \cite{MB}) gives us
	\begin{align*}
		u_s(x)&= \left( \frac{1}{\abs{x}} \int_0^{\abs{x}} r w_s(r) \, dr \right) \mathbf{e}_r \eqqcolon G(\abs{x}) \mathbf{e}_\theta , \\
		\nabla u_s (x) &= \frac{G(\abs{x})}{\abs{x}} \begin{pmatrix}
			0 & -1 \\
			1 & 0
		\end{pmatrix} + \left( G'(\abs{x}) - \frac{G(\abs{x})}{\abs{x}}\right) \mathbf{e}_\theta \mathbf{e}_r^t,
	\end{align*}
	from which the result is obvious\footnote{Since $w_s$ is radial, $w_s(r)$ is well-defined without confusion.}.
\end{proof}

Now, we give more refined rate of access of two points than Lemma \ref{access rate} using upgraded regularity of $g$ obtained from Theorem \ref{main thm 1}.

\begin{lemma} \label{access rate alpha lemma}
	Let $\alpha \in (0,1)$ be given and $\Phi: \bbR^2 \times [0,1] \rightarrow \bbR^2$ be a flow map of $u$. Then, there exist $M>0$\footnote{Originally, the constant depends on $\cpsi$- norm of the solution $u$. However, we declared in Section \ref{subsec: Notations} that as we fix $u$ as the specific solution of \eqref{equ1}, dependence on $u$ would not be mentioned.}, $ 0<R=R(\alpha)<e^{-3}$ such that for all $0<\abs{x-y}<R$ and $0 \leq t \leq 1$,
	\begin{equation} \label{access rate alpha}
		\abs{x-y} e^{-Mt \logonealpha{x-y}} \leq \abs{\Phi(x,t)-\Phi(y,t)} \leq \abs{x-y} e^{Mt \logonealpha{x-y}}.
	\end{equation}
\end{lemma}

\begin{proof}
	Since $g_0$ is compactly supported smooth function, $g_0$ also belongs to $\cphi$ regardless of $\alpha$. Then, $g\in L^\infty([0,1]; \cphi)$ by Theorem \ref{main thm 1} and $u_g \in L^\infty([0,1];  \cpsi)$ by Lemma \ref{vor to vel lemma}. Thus $u \in L^\infty([0,1]; \cpsi)$ and let
	\begin{equation*}
		\sup_{0 \leq t \leq 1} \nrm{u(\cdot,t)}_{\cpsi} \eqqcolon N.
	\end{equation*}
	Then, comparison with ODE
	\begin{equation*}
		\left\{ \begin{aligned}
			& x'(t)= \pm N x(t) \left( \log \frac{1}{x(t)} \right)^{1-\alpha}, \\
			& x(0)=x_0 \quad (0<x_0<e^{-3(N+1)}),
		\end{aligned} \right.
	\end{equation*}
	whose solution is
	\begin{equation*}
		x(t)=e^{-\left(\mp \alpha Nt+ \left( \log \frac{1}{x_0}\right)^\alpha\right)^{1/\alpha}},
	\end{equation*}
	gives us
	\begin{equation} \label{eq12}
		e^{-\left( \alpha Nt+ \left( \log \frac{1}{\abs{x-y}}\right)^\alpha\right)^{1/\alpha}} \leq \abs{\Phi(x,t)-\Phi(y,t)} \leq e^{-\left(- \alpha Nt+ \left( \log \frac{1}{\abs{x-y}}\right)^\alpha\right)^{1/\alpha}}
	\end{equation}
	for $0<\abs{x-y}<e^{-3(N+1)}$. Applying elementary inequalities
	\begin{equation*}
		(a+b)^k \leq a^k +bk(a+b)^{k-1}, \quad (a-b)^k  \geq a^k -bka^{k-1} \quad (a,b,k>0)
	\end{equation*}
	to \eqref{eq12}, we have
	\begin{equation} \label{eq13}
		e^{-\log \frac{1}{\abs{x-y}} -Nt \left( \alpha Nt+ \left( \log \frac{1}{\abs{x-y}} \right)^\alpha \right)^{1/\alpha -1}} \leq \abs{\Phi(x,t)-\Phi(y,t)} \leq e^{-\log \frac{1}{\abs{x-y}}+Nt \logonealpha{x-y}}.
	\end{equation}
	If $\abs{x-y}<e^{-\left( \frac{\alpha N}{2^\alpha -1} \right)^{1/\alpha}}$, then $\alpha Nt \leq (2^\alpha-1) \left( \log \frac{1}{\abs{x-y}}\right)^\alpha$ so that \eqref{eq13} could be simplified to
	\begin{equation*}
		\abs{x-y} e^{-2Nt \logonealpha{x-y}} \leq \abs{\Phi(x,t)-\Phi(y,t)} \leq \abs{x-y} e^{Nt \logonealpha{x-y}}.
	\end{equation*}
\end{proof}

Now, we estimate the value of $g$. The inequality \eqref{eq9} implies that there exists a constant $C=C(\alpha)$ such that for $0 \leq t \leq 1$,
\begin{equation} \label{eq14}
	\nrm{g(\cdot,t)}_{\cphi} \leq \nrm{g_0}_{\cphi} + Ct.
\end{equation}
However, we need to get sharper time-dependent local estimates near the origin.

\begin{lemma} \label{local g estimate}
	Let $\alpha \in (0,1)$ be given. Then, there exist $R>0, M=M(\alpha)>0$ and $0<T\leq 1$ such that for all $0<\abs{x}<R$ and $0 \leq t \leq T$,
	\begin{equation} \label{g}
		\abs{g(x,t)} \leq M \logalpha{x} t,
	\end{equation}
	\begin{equation} \label{ug}
		\abs{u_g(x,t)} \leq M\abs{x} \left( 1+ \logonealpha{x}t \right).
	\end{equation}
\end{lemma}

\begin{proof}
	Since $g\in \cphi$ and $u_g(0,t)=0$ due to symmetry, Proposition \ref{consistency lemma} could be applied. As $g$ satisfies \eqref{equation for g}, we have
	\begin{equation} \label{eq15}
		g(x,t) = \underbrace{g_0(\Phi_t^{-1}(x))}_{\eqqcolon F(x,t)} - \underbrace{\int_0^t u_g(\Phi(\Phi_t^{-1}(x),s)) \cdot \nabla w_s(\Phi(\Phi_t^{-1}(x),s))\, ds}_{\eqqcolon G(x,t)}, 
	\end{equation}
	where $\Phi: \bbR^2 \times [0,1] \rightarrow \bbR^2$ is the flow map generated by $u$ and $\Phi_t^{-1}$ is the inverse flow map. Since $\sup_{0 \leq t \leq 1} \nrm{u}_{LL}<\infty$, there exits $0<T<1$ such that if $0\leq t \leq T$, then $\abs{\Phi(x,t)-x}\leq 1, \abs{\Phi_t^{-1}(x)-x}\leq 1$ for all $x \in \bbR^2$. Then, considering the fact that $g_0$ is supported in $\bbR^2 \backslash B(0,21)$ and $w_s$ is supported in $B(0,1)$ (see \eqref{initial data} and \eqref{def of ws}),
	\begin{equation*}
		\left\{ \begin{aligned}
			& \supp(F(\cdot,t) ) \subset \bbR^2 \backslash B(0,20), \\
			& \supp (G(\cdot,t) ) \subset B(0,2),
		\end{aligned} \right.
	\end{equation*}
	for $0 \leq t \leq T$. Using same method in the proof of Proposition \ref{consistency lemma}, there exists $M=M(\alpha)>0$ such that
	\begin{equation*}
		\nrm{ u_g(\Phi(\Phi_t^{-1}(x),s)) \cdot \nabla w_s(\Phi(\Phi_t^{-1}(x),s)) }_{\cphi} \leq M \nrm{u_g(\cdot,t)}_{\cpsi} 
		\underset{\eqref{vor to vel} }{\leq} M \nrm{g(\cdot,t)}_{\cphi}
		\underset{\eqref{eq14} }{\leq} \nrm{g_0}_{\cphi}+Ct.
	\end{equation*}
	Therefore, there exists a constant $M=M(\alpha)>0$ such that
	\begin{equation} \label{equ12}
		\nrm{G(\cdot,t)}_{\cphi} \leq Mt,\
	\end{equation}
	for $0 \leq t \leq T$, which implies inequality \eqref{g}. Then, we could get \eqref{ug} through the decomposition \eqref{eq15} by applying \eqref{equ12} with Lemma \ref{vor to vel} and the fact that $F(\cdot,t)$ generates Lipschitz velocity field near the origin.
\end{proof}

Using Lemma \ref{local g estimate}, we give an estimate for $\nabla u$.

\begin{proposition} \label{del u proposition}
	Let $\alpha \in (0,1)$ be given. Then, there exist $0<R=R(\alpha)<e^{-3} , M>0$, and $0 \leq T \leq 1$ such that for all $0< \abs{x} <R$ and $0 \leq t \leq T$,
	\begin{equation} \label{del u}
		\abs{\nabla u(x,t)} \leq M \left( \log \log \frac{1}{\abs{x}}+ \logonealpha{x}t + e^{Mt \logonealpha{x}}   \right).
	\end{equation}
\end{proposition}

\begin{proof}
	Application of interior Schauder estimate (see, e.g., Theorem 2.14 of \cite{Schauder-Xavier}) to relation
	\begin{equation*}
		\Delta \psi =w \quad \text{in} \quad B\left(x,\frac{\abs{x}}{2}\right),
	\end{equation*}
	where $\psi:\bbR^2 \rightarrow \bbR$ is a stream function such that $\nabla^\perp \psi =u$, and $\psi(0)=0$, yields that there exists $C>0$ such that
	\begin{equation} \label{Shauder}
		\nrm{\nabla u(\cdot,t)}_{L^\infty \left(B \left(x, \frac{\abs{x}}{4}\right)\right)} \leq C  \left(  \frac{1}{\abs{x}^2} \nrm{\psi(\cdot,t)}_{L^\infty \left(B \left(x, \frac{\abs{x}}{2}\right)\right)} + \nrm{w(\cdot,t)}_{L^\infty \left(B \left(x, \frac{\abs{x}}{2}\right)\right)} + \abs{x}^\alpha [w(\cdot,t)]_{C^{0,\alpha} \left(B \left(x, \frac{\abs{x}}{2}\right)\right)}  \right).
	\end{equation}

We first estimate $\nrm{\psi(\cdot,t)}_{L^\infty \left(B \left(x, \frac{\abs{x}}{2}\right)\right)}$. By \eqref{us} and \eqref{ug}, there exists $0 < R < e^{-4}, M=M(\alpha)>0$, and $T>0$ such that for all $0<\abs{x}<R$ and $0 \leq t \leq T$,
\begin{equation} \label{u}
	\abs{u(x,t)} \leq \abs{u_s(x)}+ \abs{u_g(x,t)} \underset{\eqref{us}, \eqref{ug} }{\leq} \abs{x} \log \log \frac{1}{\abs{x}} + M \abs{x} \logonealpha{x}t.
\end{equation}
As $u=\nabla^\perp \psi$ and $\psi(0)=0$, the Fundamental Theorem of Calculus implies that
\begin{equation} \label{eq16}
	\begin{aligned}
		\nrm{\psi(\cdot,t)}_{L^\infty \left(B \left(x, \frac{\abs{x}}{2}\right)\right)} & \leq \frac{3}{2} \abs{x} \sup_{0 \leq \abs{z} \leq \frac{3}{2}\abs{x}} \abs{u(z,t)} \\
		\underset{\eqref{us}, \eqref{ug} }&{\leq} M \abs{x} \left( \abs{x} \log \log \frac{1}{\abs{x}} + M \abs{x} \logonealpha{x}t \right).
	\end{aligned}
\end{equation}

Also, by \eqref{def of ws} and \eqref{g}, there exist $0<R=R(\alpha) < e^{-4},M=M(\alpha)>0$ and $T>0$ such that for all $0<\abs{x}<R$ and $0 \leq t \leq T \leq 1$ (we can choose $R=R(\alpha)$ so that $M\left( \log \frac{1}{R}\right)^{-\alpha} \leq \log \log \frac{1}{R}$ holds),
\begin{equation} \label{eq17}
	\abs{w(x,t)} \underset{\text{Lem } \ref{symmetry lemma}}{\leq} \abs{w_s(x)}+\abs{g(x,t)} \underset{\eqref{def of ws}, \eqref{g} }{\leq} \log \log \frac{1}{\abs{x}} + M\logalpha{x}t \leq 2 \log \log \frac{1}{\abs{x}}.
\end{equation}

Lastly, we calculate H\"older seminorm $[w(\cdot,t)]_{C^{0,\alpha} \left(B \left(x, \frac{\abs{x}}{2}\right)\right)}$. Suppose two points $y, z \in B\left( x, \frac{\abs{x}}{2} \right)$ are given. Then, for all $\abs{x}<R$ and $0 \leq t \leq T$, where $R$ is a constant in Lemma \ref{access rate alpha lemma} and $T$ is the constant in Lemma \ref{local g estimate},

\begin{equation*} 
	\begin{aligned}
	\frac{\abs{w(y,t)-w(z,t)}}{\abs{y-z}^\alpha} &= \frac{\abs{w_0(\Phi_t^{-1}(y))-w_0(\Phi_t^{-1}(z))}}{\abs{y-z}^\alpha} \\
	\underset{\text{Lem } \ref{symmetry lemma}}&{\leq} \frac{|w_s(\Phi_t^{-1}(y))-w_s(\Phi_t^{-1}(z))+\overbrace{g_0(\Phi_t^{-1}(y))-g_0(\Phi_t^{-1}(z))}^{=0 ~~(\because \, \supp(g_0) \subset \bbR^2 \backslash B(0,2))}|}{\abs{y-z}^\alpha} \\
	& = \frac{|w_s(\Phi_t^{-1}(y))-w_s(\Phi_t^{-1}(z))|}{\abs{\Phi_t^{-1}(y) - \Phi_t^{-1}(z)}} \cdot \frac{\abs{\Phi_t^{-1}(y) - \Phi_t^{-1}(z)}}{\abs{y-z}^\alpha}.
\end{aligned}
\end{equation*}

By mean value theorem,
\begin{equation} \label{equ2}
	\begin{aligned}
	\sup_{y,z\in B\left( x, \frac{\abs{x}}{2}\right)} \frac{|w_s(\Phi_t^{-1}(y))-w_s(\Phi_t^{-1}(z))|}{\abs{\Phi_t^{-1}(y) - \Phi_t^{-1}(z)}} &\leq \sup_{p \in \Phi_t^{-1}\left( B \left( x, \frac{\abs{x}}{2}\right)\right)} \abs{\nabla w_s(p)} \\
	\underset{\text{Lem } \ref{access rate alpha lemma} }&{\leq} \abs{\nabla w_s \left( \frac{\abs{x}}{2} e^{-Mt \left( \log \frac{1}{\abs{x}/2}\right)^{1-\alpha}}\right)} \\
	& \leq \frac{2e^{2^{1-\alpha}Mt \logonealpha{x}}}{\abs{x} \log \frac{1}{\abs{x}}}.
\end{aligned}
\end{equation}

Also, by Lemma \ref{access rate alpha lemma},
\begin{equation*}
	\frac{\abs{\Phi_t^{-1}(y) - \Phi_t^{-1}(z)}}{\abs{y-z}^\alpha} \leq \abs{y-z}^{1-\alpha} e^{Mt \logonealpha{y-z}}.
\end{equation*}
Note that for one-variable function $f(r)=r^{1-\alpha}e^{Mt\left( \log \frac{1}{r}\right)^{1-\alpha}}$,
\begin{equation*}
	f'(r)=(1-\alpha)r^{-\alpha}e^{Mt \left( \log \frac{1}{r}\right)^{1-\alpha}} \left( 1-Mt \left( \log \frac{1}{r} \right)^{-\alpha} \right).
\end{equation*}
This means that the supremum of $\abs{y-z}^{1-\alpha} e^{Mt \logonealpha{y-z}}$ for $y, z \in B \left( x, \frac{\abs{x}}{2} \right)$ is achieved when $\abs{y-z}=\abs{x}$ for sufficiently small $x$ (precisely, $0<\abs{x}<R$ suffices for $R$ satisfying $M \left( \log \frac{1}{R}\right)^{-\alpha} <\frac{1}{2}$). Then, we have
\begin{align*}
	\sup_{y,z\in B\left( x, \frac{\abs{x}}{2}\right)} \frac{\abs{\Phi_t^{-1}(y) - \Phi_t^{-1}(z)}}{\abs{y-z}^\alpha} & \leq \sup_{y,z\in B\left( x, \frac{\abs{x}}{2}\right)}  \abs{y-z}^{1-\alpha} e^{Mt \logonealpha{y-z}} \\
	& = \abs{x}^{1-\alpha} e^{Mt \logonealpha{x}}.
\end{align*}
Therefore, there exists $R=R(\alpha)>0, T>0$ and $M>0$ such that for $0<\abs{x}<R$ and $0 \leq t \leq T$,
\begin{equation} \label{eq18}
	\begin{aligned}
		\sup_{y,z\in B\left( x, \frac{\abs{x}}{2}\right)} \frac{\abs{w(y,t)-w(z,t)}}{\abs{y-z}^\alpha} & \leq \sup_{y,z\in B\left( x, \frac{\abs{x}}{2}\right)} \frac{|w_s(\Phi_t^{-1}(y))-w_s(\Phi_t^{-1}(z))|}{\abs{\Phi_t^{-1}(y) - \Phi_t^{-1}(z)}} \cdot \sup_{y,z\in B\left( x, \frac{\abs{x}}{2}\right)} \frac{\abs{\Phi_t^{-1}(y) - \Phi_t^{-1}(z)}}{\abs{y-z}^\alpha} \\
		& \leq 2 \frac{e^{2^{1-\alpha}Mt\logonealpha{x}}}{\abs{x} \log \frac{1}{\abs{x}}} \abs{x}^{1-\alpha} e^{Mt\logonealpha{x}} \\
		& \leq \frac{2e^{3Mt\logonealpha{x}}}{\abs{x}^\alpha \log \frac{1}{\abs{x}}}.
	\end{aligned}
\end{equation}
Combining \eqref{eq16}, \eqref{eq17}, \eqref{eq18} with \eqref{Shauder}, we conclude that there exists $R=R(\alpha)>0, T>0$ and $M>0$ such that for $0<\abs{x}<R$ and $0 \leq t \leq T$,
\begin{equation*}
	\abs{\nabla u(x,t)} \leq \nrm{\nabla u (\cdot,t)}_{L^\infty \left( B \left( x, \frac{\abs{x}}{4} \right)\right)} \underset{\substack{\eqref{Shauder}, \eqref{eq16} \\ \eqref{eq17}, \eqref{eq18}  }}{\leq} M  \left( \log \log \frac{1}{\abs{x}} + \logonealpha{x}t + e^{Mt \logonealpha{x}}\right).
\end{equation*}
\end{proof}

\subsection{The Key Lemma}
\label{subsec: The Key Lemma}

Remember that the evolution equation for $g \in L^\infty_{\textnormal{loc}}([0,\infty); L^\infty)$ is given by \eqref{equation for g}
\begin{equation*}
	g_t + u \cdot \nabla g = -u_g \cdot \nabla w_s.
\end{equation*}
Since $w_s(x)=\log \log \frac{1}{\abs{x}}$ for $0<\abs{x} <e^{-2}$, the equation for $g$ near the origin could be written by
\begin{equation} \label{equation for g near the origin}
	g_t+u \cdot \nabla g = \frac{u_g \cdot \frac{x}{\abs{x}}}{\abs{x} \log \frac{1}{\abs{x}}}.
\end{equation}
Note that \textit{when time is zero}, \eqref{Lip initial} implies that there exists $K, \delta>0$ such that for $x=(0,r)$ with $0<r<\delta$,
\begin{equation*}
	\frac{\abs{u_g(x,0) \cdot \frac{x}{\abs{x}}}}{\abs{x} \log \frac{1}{\abs{x}}} = \frac{\abs{u_{g_0}(0,r) \cdot (0,1) }}{r \log \frac{1}{r}} \geq \frac{K}{\boldsymbol{\log \frac{1}{r}}}.
\end{equation*}
This implies that forcing term on the right hand side of \eqref{equation for g near the origin} is so strong enough that one could expect that $g$ might not belong to $L^\infty([0,1]; C^{\phi_\beta})$ for any $\beta>1$ (remember $\phi_\beta(r) = \left( \log \frac{1}{r}\right)^{-\beta}$). To make this argument rigorous, what we have to do is to remove the qualifier `\textit{when time is zero}' so that this large forcing maintains for finite length of time interval. This is what our Key Lemma exactly tells about.

\begin{definition}
	A map $\phi_r:[0,\infty) \rightarrow\bbR^2$ is defined by particle trajectory of $(0,r)$, i.e., $\phi_r$ is a unique solution of ODE
	\begin{equation*}
		\left\{ \begin{aligned}
			& \frac{d}{dt} \phi_r(t)=u(\phi_r(t),t), \\
			&\phi_r(0)=(0,r),
		\end{aligned} \right.
	\end{equation*}
	where $u$ is fixed solution of \eqref{equ1}.
\end{definition}

By the definition of $w_s$ \eqref{def of ws} and Lemma \ref{access rate alpha lemma}, there exist $R=R(\alpha)>0$ and $0 \leq t \leq T$ such that forced transport equation \eqref{equation for g} could be written along particle trajectory $\phi_r$
\begin{equation} \label{equation for g integral}
	g(\phi_r(t),t) = \int_0^t \frac{u_g(\phi_r(s),s) \cdot \frac{\phi_r(s)}{\abs{\phi_r(s)}}}{\abs{\phi_r(s)} \log \frac{1}{\abs{\phi_r(s)}}} \, ds
\end{equation}
for all $0<r<R$ and $0 \leq t \leq T$. We now state the Key Lemma.

\begin{lemma} [\textbf{Key Lemma}] \label{key lemma}
	Let $\alpha \in (0,1)$ be given. Then, for all $\varepsilon, \delta>0$, there exists $R=R(\alpha, \varepsilon)>0$ such that for $0<r<R$ and $0 \leq t \leq \frac{1}{\left( \log \frac{1}{r} \right) ^{1-\alpha +\varepsilon}}$,
	\begin{equation*}
		\abs{u_g(\phi_r(t),t)-u_g(\phi_r(0),0)} \leq \delta r.
	\end{equation*}
\end{lemma}
\vspace{2mm}

Since $g_0$ is compactly supported smooth function, $g_0 \in \cphi$ for given $\alpha \in (0,1)$ so that results in the Section \ref{sec: Conservation of Modulus of Continuity} could be applied. From
\begin{equation*}
	w_s(\Phi_t^{-1}(x))+g_0(\Phi_t^{-1}(x)) \underset{\eqref{equ1}}{=}w_0(\Phi_t^{-1}(x)) \underset{\eqref{equ1}}{=} w(x,t) \underset{\text{Lem }\ref{symmetry lemma} }{=} w_s(x)+g(x,t),
\end{equation*}
$g(x,t)$ could be written as
\begin{equation} \label{eq19}
	g(x,t)=\Big(w_s(\Phi_t^{-1}(x))-w_s(x)\Big)+g_0(\Phi_t^{-1}(x)).
\end{equation}
Considering the support of $w_s$ and $g_0$ with $\nrm{u}_{LL}<\infty$, there exists $T>0$ such that for all $0 \leq t \leq T$, 
\begin{equation} \label{supp condition}
	\left\{ \begin{aligned}
		& \supp(w_s(\Phi_t^{-1}(\cdot))-w_s(\cdot)) \subset B(0,2), \\
		& \supp(g_0(\Phi_t^{-1}(\cdot)) \subset A(20,40).
	\end{aligned} \right.
\end{equation}
From now on, \textbf{we always take all constraint on $t$ less than this $T$} so that this support condition \eqref{supp condition} holds. Using this decomposition, we can divide $\abs{u_g(\phi_r(t),t)-u_g(\phi_r(0),0)}$ into three parts

\begin{equation} \label{eq20}
	\begin{aligned}
		\abs{u_g(\phi_r(t),t)-u_g(\phi_r(0),0)} & \leq \abs{u_g(\phi_r(t),t)-u_g(\phi_r(0),t)}+\abs{u_g(\phi_r(0),t)-u_g(\phi_r(0),0)} \\
		&=	\abs{u_g(\phi_r(t),t)-u_g(\phi_r(0),t)}+ \frac{1}{2\pi} \abs{\int_{\bbR^2} \frac{((0,r)-y)^\perp}{\abs{(0,r)-y}^2} (g(y,t)-g_0(y)) \, dy} \\
		\underset{\eqref{eq19} }&{=} \underbrace{\abs{u_g(\phi_r(t),t)-u_g(\phi_r(0),t)}}_{\eqqcolon I_1(r,t)} \\
		& \qquad + \frac{1}{2\pi}  \underbrace{\bigg\lvert\int_{B(0,2)} \frac{((0,r)-y)^\perp}{\abs{(0,r)-y}^2} (w_s(\Phi_t^{-1}(y))-w_s(y)) \, dy\bigg\rvert}_{\eqqcolon I_2(r,t)}  \\
		& \qquad + \frac{1}{2\pi}  \underbrace{\bigg\lvert\int_{\bbR^2 \backslash B(0,20)} \frac{((0,r)-y)^\perp}{\abs{(0,r)-y}^2} (g_0(\Phi_t^{-1}(y))-g_0(y)) \, dy\bigg\rvert}_{\eqqcolon I_3(r,t)} .
	\end{aligned}
	\end{equation}
We now estimate each of $I_1, I_2$ and $I_3$.

\begin{proposition} [Estimation of $I_1$] \label{estimation of I1}
	Let $\alpha \in (0,1)$ be given.  Then, there exists $R=R(\alpha)>0, 0<T\leq 1$ and $M>0$ such that for all $0<r<R$ and $0 \leq t \leq T$,
	\begin{equation*}
		I_1(r,t)=\abs{u_g(\phi_r(t),t)-u_g(\phi_r(0),t)} \leq Mtr \left( \log \frac{1}{r} \right)^{1-\alpha}e^{Mt\left( \log \frac{1}{r}\right)^{1-\alpha}}.
	\end{equation*}
\end{proposition}

\begin{proof}
As in the proof of Lemma \ref{local g estimate}, we can decompose $u_g$ into
\begin{equation*}
	u_g = \underbrace{\nabla^\perp \Delta^{-1}F}_{\eqqcolon v_1}+\underbrace{\nabla^\perp \Delta^{-1}G}_{\eqqcolon v_2},
\end{equation*}
where $F$ and $G$ are defined same as in the proof of Lemma \ref{local g estimate}. Recall that $v_1 \in L^\infty([0,1];Lip(B(0,1)))$ and $v_2 \in L^\infty([0,1]; \cpsi)$ with
\begin{equation*}
	\nrm{v_2(\cdot,t)}_{\cpsi} \leq Mt
\end{equation*}
for some constant $M=M(\alpha)>0$.

Take $0 < T \leq 1$ sufficiently small so that $\abs{\phi_r(t)-\phi_r(0)} < e^{-3}$ for all $0 \leq t \leq T$. Also, take $R>0$ sufficiently small so that $re^{M \left( \log \frac{1}{r}\right)^{1-\alpha}}<e^{-3}$ for all $0 < r \leq R$, where $M>0$ is a constant in Lemma \ref{access rate alpha lemma}. Then,
\begin{align*}
	& \abs{u_g(\phi_r(t),t)-u_g(\phi_r(0),t)} \\
	& \leq \abs{v_1(\phi_r(t),t)-v_1(\phi_r(0),t)}+ \abs{v_2 (\phi_r(t),t)} + \abs{v_2 (\phi_r(0),t)} \\
	&\leq \nrm{v_1}_{L^\infty_T Lip(B(0,1))} \abs{\phi_r(t)-\phi_r(0)} +Mt \left( \abs{\phi_r(t)} \left( \log \frac{1}{\abs{\phi_r(t)}}\right)^{1-\alpha}+r \left( \log \frac{1}{r}\right)^{1-\alpha}  \right) \\
	\underset{\text{Lem }\ref{access rate alpha} }&{\leq} \nrm{v_1}_{L^\infty_T Lip(B(0,1))} \abs{\phi_r(t)-\phi_r(0)} +Mt \left( r\left( \log \frac{1}{r} \right)^{1-\alpha}e^{Mt\left( \log \frac{1}{r}\right)^{1-\alpha}} +r \left( \log \frac{1}{r}\right)^{1-\alpha}  \right).
\end{align*}
Here, we can estimate $\abs{\phi_r(t)-\phi_r(0)}$ by
\begin{equation} \label{equ5}
	\begin{aligned}
	\abs{\phi_r(t)-\phi_r(0)} &= \abs{\int_0^t  \frac{d}{ds} \phi_r(s)\, ds} \\
	& \leq \int_0^t \abs{u(\phi_r(s),s)}\, ds \\
	\underset{\text{Lem } \ref{access rate alpha lemma} }&{\leq} t \cdot \sup_{x \in B ( 0,re^{Mt\left(\log \frac{1}{r}\right)^{1-\alpha}})} \abs{u(x)} \\
	\underset{\eqref{u} }&{\leq} t \left( re^{Mt\left(\log \frac{1}{r}\right)^{1-\alpha}} \log \log \frac{1}{r} + M r\left( \log \frac{1}{r}\right)^{1-\alpha} e^{Mt\left(\log \frac{1}{r}\right)^{1-\alpha}}\right).
	\end{aligned}
\end{equation}
	Combining these estimates leads to Proposition \ref{estimation of I1}.
\end{proof}

Before estimating $I_2$, which is the hardest part, we give the estimate for $I_3$.

\begin{proposition} [Estimation of $I_3$] \label{estimation of I3}
	Let $\alpha \in (0,1)$ be given. Then, there exist $ 0<T\leq 1$ and $M>0$ such that for all $0<r<e^{-2}$ and $0 \leq t \leq T$,
	\begin{equation*}
		I_3(r,t)=\bigg\lvert\int_{\bbR^2 \backslash B(0,20)} \frac{((0,r)-y)^\perp}{\abs{(0,r)-y}^2} (g_0(\Phi_t^{-1}(y))-g_0(y)) \, dy\bigg\rvert \leq Mtr.
	\end{equation*}
\end{proposition}

\begin{proof}
	For convenience, let
	\begin{equation*}
		g_0(\Phi_t^{-1}(y))-g_0(y) \eqqcolon h(y,t).
	\end{equation*}
	 Then, there exists $T>0$ such that for all $0 \leq t \leq T$, $\supp (h(\cdot,t)) \subset A(20,40)$. Also, since $g_0$ is Lipschitz,
	\begin{equation} \label{eq21}
		\abs{h(y,t)} \leq \nrm{g_0}_{Lip} \abs{\Phi_t^{-1}(y)-y} \leq \nrm{g_0}_{Lip} \nrm{u}_{L^\infty_T L^\infty} t
	\end{equation}
	for all $0 \leq t \leq T$. Then, as $h(y,t)=h(-y,t)$, we have
	\begin{align*}
		I_3(r,t) &=\bigg\lvert\int_{\bbR^2 \backslash B(0,20)} \frac{((0,r)-y)^\perp}{\abs{(0,r)-y}^2} h(y,t) \, dy\bigg\rvert \\
		&=\bigg\lvert\int_{A(20,40)} \frac{((0,r)-y)^\perp}{\abs{(0,r)-y}^2} h(y,t) \, dy\bigg\rvert \\
		&=  \Bigg\vert \int_{A(20,40) \cap \set{y \geq 0} } \frac{((0,r)-y)^\perp}{\abs{(0,r)-y}^2} h(y, t) \, dy+\int_{A(20, 40) \cap \set{y \leq 0}} \frac{((0,r)-y)^\perp}{\abs{(0,r)-y}^2} h(y, t) \, dy \Bigg\vert \\
		&= \abs{\int_{A(20, 40) \cap \set{y \geq 0} } \frac{((0,r)-y)^\perp}{\abs{(0,r)-y}^2} h(y, t) \, dy+\int_{A(20, 40) \cap \set{y \geq 0}} \frac{((0,r)+y)^\perp}{\abs{(0,r)+y}^2} h(-y, t) \, dy} \\
			&=\abs{\int_{A(20, 40) \cap \set{y \geq 0} } \left(\frac{((0,r)+y)^\perp}{\abs{(0,r)+y}^2}-\frac{(-(0,r)+y)^\perp}{\abs{(0,r)-y}^2}\right) h(y, t) \, dy} \\
			&\leq \text{Vol}(A(20,40)) \cdot \sup_{y \in A(20, 40)} \left(\frac{((0,r)+y)^\perp}{\abs{(0,r)+y}^2}-\frac{(-(0,r)+y)^\perp}{\abs{(0,r)-y}^2}\right)  \cdot \sup_{\substack{y \in A(20, 40) \\ 0 \leq t \leq T}} \abs{h(y,t)} \\
			\underset{\eqref{eq21} }&{\leq} \Bigg( \sup_{y \in A(20,40)} \frac{1}{\pi} \frac{2r}{\abs{(0,r)+y} \abs{(0,r)-y}} \Bigg) \cdot Mt \\
			&= Mtr.
	\end{align*}
\end{proof}

\begin{proposition} [Estimation of $I_2$] \label{estimation of I2}
	Let $\alpha \in (0,1)$ and $\varepsilon>0$ be given. Then, there exist $R=R(\alpha, \varepsilon)>0$ and $M=M(\alpha,\varepsilon)>0$ such that for all $0<r<R$ and $0 \leq t \leq \frac{1}{\left( \log \frac{1}{r}\right)^{1-\alpha+\varepsilon}}$,
	\begin{equation*}
		I_2(r,t) = \bigg\vert\int_{B(0,2)} \frac{( (0, r) -y)^\perp}{\abs{(0,r)-y }^2} (w_s(\Phi_t^{-1}(y))-w_s(y)) \, dy\bigg\vert \leq Mtr \log \log \frac{1}{r}.
	\end{equation*}
\end{proposition}

\begin{proof}
	Remember that we only consider sufficiently small $t$ so that $\supp(w_s(\Phi_t^{-1}(\cdot))-w_s(\cdot)) \subset B(0,2)$. Thus, for all $0<r<1$, we can decompose $I_2$ into two parts so that $I_2 \leq I_{21}+ I_{22}$ as following:
	\begin{align*}
		I_2(r,t)&=\bigg\vert \int_{B(0,2)} \frac{( (0, r) -y)^\perp}{\abs{(0,r)-y }^2} (w_s(\Phi_t^{-1}(y))-w_s(y)) \, dy\bigg\vert \\
		&= \bigg\vert \int_{B((0,r),3)} \frac{( (0, r) -y)^\perp}{\abs{(0,r)-y }^2} (w_s(\Phi_t^{-1}(y))-w_s(y)) \, dy\bigg\vert \\
		&\leq \underbrace{\bigg\vert \int_{B((0,r),2r)} \frac{( (0, r) -y)^\perp}{\abs{(0,r)-y }^2} (w_s(\Phi_t^{-1}(y))-w_s(y)) \, dy\bigg\vert}_{\eqqcolon I_{21}(r,t)} \\
		& \qquad + \underbrace{\bigg\vert \int_{B((0,r),3)\, \backslash \,B((0,r),2r)} \frac{( (0, r) -y)^\perp}{\abs{(0,r)-y }^2} (w_s(\Phi_t^{-1}(y))-w_s(y)) \, dy\bigg\vert}_{\coloneqq I_{22}(r,t)}.
	\end{align*}
	
	\bigskip

	$\bullet$ \underline{Estimation of $I_{21}$}
	
	\begin{align*}
		I_{21}(r,t) &= \bigg\vert \int_{B((0,r),2r)} \frac{( (0, r) -y)^\perp}{\abs{(0,r)-y }^2} (w_s(\Phi_t^{-1}(y))-w_s(y)) \, dy\bigg\vert \\
		& \leq \sup_{y \in B((0,r),2r)} \vert w_s(\Phi_t^{-1}(y))-w_s(y)\vert \cdot \int_{B((0,r),2r)} \frac{( (0, r) -y)^\perp}{\abs{(0,r)-y }^2} \, dy.
	\end{align*}
	Combining the equality \eqref{eq19} and the support information on $g_0(\Phi_t^{-1}(\cdot))$ \eqref{supp condition} with Lemma \ref{local g estimate}, there exist constants $R>0, M=M(\alpha)>0$ and $T>0$ such that for all $0<r<R$ and $0 \leq t \leq T$,
	\begin{align*}
		\sup_{y \in B((0,r),2r)} \abs{w_s(\Phi_t^{-1}(y))-w_s(y)} \underset{\eqref{eq19} }&{=} \sup_{y \in B((0,r),2r)} \abs{g(x,t)-g_0(\Phi_t^{-1}(x))} \\
		\underset{\eqref{supp condition} }&{\leq} \sup_{y \in B((0,r),2r)} \abs{g(x,t)} \\
		\underset{\eqref{g} }&{\leq}  Mt \left( \log \frac{1}{3r} \right)^{-\alpha} \\
		&\leq Mt \left( \log \frac{1}{r} \right)^{-\alpha}.
	\end{align*}
	Therefore, as $\int_{B((0,r),2r)} \frac{( (0, r) -y)^\perp}{\abs{(0,r)-y }^2} \, dy=4\pi r,$ there exists $R>0$ and $T>0$ such that for all $0<r<R$ and $0 \leq t \leq T$, 
	\begin{equation} \label{equ4}
		I_{21}(r,t) \leq Mtr \left( \log \frac{1}{r} \right)^{-\alpha}.
	\end{equation}

	\hfill \underline{(End of the estimation of $I_{21}$)}

	$\bullet$ \underline{Estimation of $I_{22}$}
	
	We also divide $I_{22}$ into two parts so that $I_{22} \leq I_{221}+I_{222}$.
	
	\begin{equation} \label{equ6}
		\begin{aligned}
		&I_{22}(r,t) \\
			&= \bigg\vert\int_{B((0,r),3) \backslash B((0,r),2r)} \frac{((0,r)-y)^\perp}{\abs{(0,r)-y}^2} (w_s(\Phi_t^{-1}(y))-w_s(y)) \, dy\bigg\vert \\
			&=\bigg\vert\int_{A(2r,3)} \frac{(-y)^\perp}{\abs{-y}^2} (w_s(\Phi_t^{-1}((0,r)+y))-w_s((0,r)+y)) \, dy\bigg\vert \\
			&= \bigg\vert\int_{A(2r,3) \cap \set{y \geq 0}} \frac{(-y)^\perp}{\abs{-y}^2} (w_s(\Phi_t^{-1}((0,r)+y))-w_s((0,r)+y)) \, dy \\
			& \qquad +\int_{A(2r,3)\cap \set{y \leq 0}} \frac{(-y)^\perp}{\abs{-y}^2} (w_s(\Phi_t^{-1}((0,r)+y))-w_s((0,r)+y)) \, dy\bigg\vert \\
			&= \bigg\vert\int_{A(2r,3) \cap \set{y \geq 0}} \frac{(-y)^\perp}{\abs{-y}^2} (w_s(\Phi_t^{-1}((0,r)+y))-w_s((0,r)+y)) \, dy \\
			& \qquad +\int_{A(2r,3)\cap \set{y \geq 0}} \frac{y^\perp}{\abs{y}^2} (\underbrace{w_s(\Phi_t^{-1}((0,r)-y))-w_s((0,r)-y)}_{w_s(\Phi_t^{-1}(y-(0,r))) - w_s(y-(0,r))}) \, dy\bigg\vert \\
			&=\bigg\vert \int_{A(2r,3) \cap \set{y \geq 0}} \frac{y^\perp}{\abs{y}^2} \bigg( \Big( w_s(\Phi_t^{-1}(y-(0,r)))-w_s(y-(0,r)) \Big)- \Big(w_s(\Phi_t^{-1}(y+(0,r)))-w_s(y+(0,r)) \Big)  \bigg) \, dy \bigg\vert \\
			& \leq  \underbrace{\bigg\vert \int_{\substack{A(2r,R) \cap  \set{y \geq 0}}}\frac{y^\perp}{\abs{y}^2} \bigg( \Big( w_s(\Phi_t^{-1}(y-(0,r)))-w_s(y-(0,r)) \Big)- \Big(w_s(\Phi_t^{-1}(y+(0,r)))-w_s(y+(0,r)) \Big)  \bigg) \, dy \bigg\vert}_{\eqqcolon I_{221}(r,t)}  \\
			&~~+  \underbrace{\bigg\vert \int_{A(R,3) \cap \set{y \geq 0}} \frac{y^\perp}{\abs{y}^2} \bigg( \Big( w_s(\Phi_t^{-1}(y-(0,r)))-w_s(y-(0,r)) \Big)- \Big(w_s(\Phi_t^{-1}(y+(0,r)))-w_s(y+(0,r)) \Big)  \bigg) \, dy \bigg\vert}_{\eqqcolon I_{222}(r,t)}.
		\end{aligned}
	\end{equation}
	
	\bigskip
	
	\quad $\bullet$ \underline{Estimation of $I_{221}$}
	
	To bound $I_{221}$ using mean value theorem, we first investigate the function
	\begin{equation*}
		q(x,t) \coloneqq w_s(\Phi_t^{-1}(x))-w_s(x).
	\end{equation*}
	By chain rule, we have
	\begin{equation} \label{equ3}
		\begin{aligned}
			\nabla q(x,t) &= \nabla w_s(\Phi_t^{-1}(x)) \cdot \nabla \Phi_t^{-1}(x) - \nabla w_s (x) \\
			&= \nabla w_s (\Phi_t^{-1}(x)) (\underbrace{\nabla \Phi_t^{-1}(x)-I}_{\eqqcolon J_1(x,t)}) + \underbrace{\nabla w_s (\Phi_t^{-1}(x)) - \nabla w_s(x)}_{\eqqcolon J_2(x,t)}.
		\end{aligned}
	\end{equation}
	
	\bigskip
	
	\quad \quad $\bullet$ \underline{Estimation of $J_1$}
	
	By differentiating an identity
	\begin{equation*}
		\Phi_t^{-1}(\Phi(x,t)) =x,
	\end{equation*}
	with respect to space variable, we have
	\begin{equation*}
		\nabla \Phi_t^{-1}(\Phi(x,t)) \cdot \nabla \Phi(x,t) =I.
	\end{equation*}
	Replacing $x$ with $\Phi_t^{-1}(x)$ leads to
	\begin{equation*}
		\nabla \Phi_t^{-1}(x)= (\nabla \Phi (\Phi_t^{-1}(x),t))^{-1}.
	\end{equation*}
	Therefore,
	\begin{equation} \label{eq22}
		J_1(x,t)= \nabla \Phi^{-1}_t (x) - I = \left( \nabla \Phi ( \Phi^{-1}_t (x), t) \right) ^{-1} -I=\left( \nabla \Phi ( \Phi^{-1}_t (x), t) \right)^{-1} (I- \nabla \Phi (\Phi^{-1}_t (x),t) ).
	\end{equation}
	We now derive an equation for $\nabla \Phi (\Phi_t^{-1}(x),t)$. Replacing $x$ with $\Phi_t^{-1}(x)$ after taking nabla on both sides of the identity
	\begin{equation*}
		\frac{d}{dt} \Phi (x,t) = u( \Phi (x,t),t),
	\end{equation*}
	we get
	\begin{equation} \label{eq23}
		\frac{d}{dt} \nabla \Phi ( \Phi^{-1}_t (x), t) = \nabla u(x,t) \cdot \nabla \Phi (\Phi^{-1}_t (x),t).
	\end{equation}
	Applying Gr\"onwall inequality to \eqref{eq23} yields
	\begin{equation} \label{eq24}
		\begin{aligned}
			\abs{\nabla (\Phi_t^{-1}(x),t)} & \leq e^{\int_0^t \nabla u(x,s) \, ds} \\
			\underset{\eqref{del u}}&{\leq} e^{Mt\left( \log \log \frac{1}{\abs{x}}+ \left( \log \frac{1}{\abs{x}} \right)^{1-\alpha}t+e^{Mt\left( \log \frac{1}{\abs{x}}\right)^{1-\alpha}}  \right) }.
		\end{aligned}
	\end{equation}
	Then, we have
	\begin{equation} \label{eq25}
		\begin{aligned}
			\abs{(\nabla \Phi (\Phi_t^{-1}(x),t))^{-1}} & = \bigg\vert \frac{1}{\det (\nabla \Phi (\Phi_t^{-1}(x),t))} \cdot \text{adj}\,(\nabla \Phi (\Phi_t^{-1}(x),t)) \bigg\vert\\
			& = \abs{\nabla \Phi (\Phi_t^{-1}(x),t)} \\
			\underset{\eqref{eq24}}&{\leq} e^{Mt\left( \log \log \frac{1}{\abs{x}}+ \left( \log \frac{1}{\abs{x}} \right)^{1-\alpha}t+e^{Mt\left( \log \frac{1}{\abs{x}}\right)^{1-\alpha}}  \right)}.
		\end{aligned}
	\end{equation}
	In the second equality, we use incompressibility of velocity field.
	Also,
	\begin{equation} \label{eq26}
		\begin{aligned}
			& \abs{I-\Phi (\Phi_t^{-1}(x),t)} \\
			&= \bigg\vert \int_0^t \frac{d}{ds} \Phi (\Phi_s^{-1}(x),s)\, ds  \bigg\vert \\
			\underset{\eqref{eq23} }&{=} \bigg\vert \int_0^t \nabla u(x,s) \cdot \nabla \Phi (\Phi_s^{-1}(x),s) \, ds  \bigg\vert \\
			\underset{ \eqref{del u}, \eqref{eq24} }&{\leq} Mt\left( \log \log \frac{1}{\abs{x}}+ \left( \log \frac{1}{\abs{x}} \right)^{1-\alpha}t+e^{Mt\left( \log \frac{1}{\abs{x}}\right)^{1-\alpha}}  \right)e^{Mt\left( \log \log \frac{1}{\abs{x}}+ \left( \log \frac{1}{\abs{x}} \right)^{1-\alpha}t+e^{Mt\left( \log \frac{1}{\abs{x}}\right)^{1-\alpha}}  \right)}.
		\end{aligned}
	\end{equation}
	Applying estimates \eqref{eq25}, \eqref{eq26} to \eqref{eq22} implies that there exist $R=R(\alpha)>0, 0 < T \leq 1$ and $M>0$ such that for all $0< \abs{x} < R$ and $0 \leq t \leq T$,
	\begin{equation} \label{eq27}
		\begin{aligned}
			\abs{J_1(x,t)} \underset{\eqref{eq22} }&{\leq} \abs{(\nabla \Phi(\Phi_t^{-1}(x),t))^{-1}} \cdot \abs{I-\nabla \Phi (\Phi_t^{-1}(x),t)} \\
			\underset{\eqref{eq25}, \eqref{eq26} }&{\leq} Mt\left( \log \log \frac{1}{\abs{x}}+ \left( \log \frac{1}{\abs{x}} \right)^{1-\alpha}t+e^{Mt\left( \log \frac{1}{\abs{x}}\right)^{1-\alpha}}  \right)e^{Mt\left( \log \log \frac{1}{\abs{x}}+ \left( \log \frac{1}{\abs{x}} \right)^{1-\alpha}t+e^{Mt\left( \log \frac{1}{\abs{x}}\right)^{1-\alpha}}  \right)}.
		\end{aligned}
	\end{equation}
	
	\hfill \underline{(End of the estimation of $J_1$)}
	
	\bigskip
	\vspace{4.9mm}
	
	\quad \quad $\bullet$ \underline{Estimation of $J_2$}
	
	Let $M>0$ be a constant in Lemma \ref{access rate alpha lemma} and take $0<R<e^{-3}$ such that $Re^{Mt\left( \log \frac{1}{R}\right)^{1-\alpha}}\leq \min \left( e^{-3},R' \right)$ for all $0 \leq t \leq T$, where $R'>0$ is a constant which makes inequality \eqref{u} hold. then, by mean value theorem and Lemma \ref{access rate alpha lemma},
	\begin{equation} \label{eq28}
		\begin{aligned}
			\abs{J_2(x,t)} &= \abs{\nabla w_s (\Phi_t^{-1}(x)-\nabla w_s(x)} \\
			\underset{\text{Lem } \ref{access rate alpha lemma} }&{\leq} \sup_{\abs{x}e^{-Mt \left( \log \frac{1}{\abs{x}}\right)^{1-\alpha}}<\abs{y}<\abs{x}e^{Mt \left( \log \frac{1}{\abs{x}}\right)^{1-\alpha}}} \abs{\nabla^2 w_s(y)} \cdot t \sup_{\abs{x}e^{-Mt \left( \log \frac{1}{\abs{x}}\right)^{1-\alpha}}<\abs{y}<\abs{x}e^{Mt \left( \log \frac{1}{\abs{x}}\right)^{1-\alpha}}} \abs{u(y,t)} \\
			\underset{\eqref{def of ws}, \eqref{u}}&{\leq} \frac{e^{2Mt \left( \log \frac{1}{\abs{x}}\right)^{1-\alpha}}}{\abs{x}^2 \log \frac{1}{\abs{x}}} \left( Mt \abs{x}e^{Mt \left( \log \frac{1}{\abs{x}}\right)^{1-\alpha}}\left( \log \log \frac{1}{\abs{x}} +  \left( \log \frac{1}{\abs{x}}\right)^{1-\alpha}t\right)\right) \\
			&=\frac{Mte^{Mt\left( \log \frac{1}{\abs{x}}\right)^{1-\alpha}}}{\abs{x}\log \frac{1}{\abs{x}}} \left( \log \log \frac{1}{\abs{x}} +  \left( \log \frac{1}{\abs{x}}\right)^{1-\alpha} t\right)
		\end{aligned}
	\end{equation}
	for all $0<\abs{x} < R$.
		
	\hfill \underline{(End of the estimation of $J_2$)}
	
	\bigskip

	Applying inequalities \eqref{equ2}, \eqref{eq27}, and \eqref{eq28} to \eqref{equ3}, we have
	\begin{equation} \label{eq29}
		\begin{aligned}
		&\abs{\nabla q(x,t)} \\
			 \underset{\eqref{equ3} }&{\leq }  \abs{\nabla w_s (\Phi_t^{-1}(x))} \cdot \abs{\nabla \Phi_t^{-1}(x)-I} + \abs{\nabla w_s (\Phi_t^{-1}(x)) - \nabla w_s(x)} \\
			\underset{\substack{\eqref{equ2}, \eqref{eq27}\\ \eqref{eq28} }}&{\leq} \frac{Mt}{\abs{x} \log \frac{1}{\abs{x}}}\left( \log \log \frac{1}{\abs{x}}+ \left( \log \frac{1}{\abs{x}} \right)^{1-\alpha}t+e^{Mt\left( \log \frac{1}{\abs{x}}\right)^{1-\alpha}}  \right) e^{Mt\left( \log \log \frac{1}{\abs{x}}+ \left( \log \frac{1}{\abs{x}} \right)^{1-\alpha}+e^{Mt\left( \log \frac{1}{\abs{x}}\right)^{1-\alpha}}  \right)},
		\end{aligned}
	\end{equation}
	whenever $0 < \abs{x} < R$ and $0 \leq t \leq T$ for some constant $R=R(\alpha)>0, M>0$ and $T>0$. In particular, for any $\varepsilon>0$, we can take $R=R(\alpha,\varepsilon)<e^{-3}$ such that for all $0<\abs{x}<2R$,
	\begin{equation*}
		\frac{1}{\left( \log \frac{1}{\abs{x}}\right)^{1-\alpha+\varepsilon }} \left( \log \log \frac{1}{\abs{x}} + \left( \log \frac{1}{\abs{x}}\right)^{1-\alpha}\right) < \frac{1}{100(M+1)^2}
	\end{equation*}
	holds. Then, for all $0<\abs{x}<2R$ and $0 \leq t \leq \frac{1}{\left( \log \frac{1}{\abs{x}}\right)^{1-\alpha+\varepsilon}}$, we have
	\begin{equation} \label{eq30}
		\abs{\nabla q(x,t)} \leq \frac{3Mt}{\abs{x} \log \frac{1}{\abs{x}}} \log \log \frac{1}{\abs{x}}.
	\end{equation}
	Then, for all $0<r<\frac{R}{2}$ and $0 \leq t \leq \frac{1}{\left( \log \frac{1}{r} \right)^{1-\alpha+\varepsilon} }$,
	\begin{equation} \label{eq31}
		\begin{aligned}
			&I_{221}(r,t)\\
			&=\bigg\vert \int_{A(2r,R) \cap \set{y \geq 0}}\frac{y^\perp}{\abs{y}^2} \bigg( \Big( w_s(\Phi_t^{-1}(y-(0,r)))-w_s(y-(0,r)) \Big)- \Big(w_s(\Phi_t^{-1}(y+(0,r)))-w_s(y+(0,r)) \Big)  \bigg) \, dy \bigg\vert \\
			&= \bigg\vert \int_{A(2r,R) \cap \set{y \geq 0}}\frac{y^\perp}{\abs{y}^2} \left(q(y-(0,r),t)-q(t+(0,r),t) \right) \, dy \bigg\vert \\
			& \leq \int_{A(2r,R) \cap \set{y \geq 0}} \frac{1}{\abs{y}} \cdot 2r \cdot \sup_{0 \leq s \leq 2} \abs{\nabla q (y-(0,(1-s)r),t) } \, dy.
		\end{aligned}
	\end{equation}
	Considering the fact that $\set{y-(0,(1-s)r) \, \vert \, y \in A(2r,R), 0 \leq s \leq 2} \subset B(r,2R)$ and $\abs{y-(0,(1-s)r)} \geq \frac{\abs{y}}{2} \text{ for } y \in A(2r,R)$, inequality \eqref{eq31} leads to
 	\begin{align*}
 			I_{221}(r,t) & \leq 2r  \int_{A(2r,R) \cap \set{y \geq 0}} \frac{1}{\abs{y}} \frac{3Mt}{\frac{\abs{y}}{2} \log \frac{1}{\abs{\frac{y}{2}}}} \log \log \frac{1}{\abs{\frac{y}{2}}} \, dy \\
			& \leq 24tr  \int_{A(2r,R) \cap \set{y \geq 0}} \frac{\log \log \frac{1}{\abs{y}}}{\abs{y}^2 \log \frac{1}{\abs{y}}} \, dy \\
			&= 24 \pi tr \int_{2r}^R \frac{\log \log \frac{1}{\rho}}{\rho \log \rho} \, d\rho \\
			& \leq 12 \pi t r \log \log \frac{1}{r}.
 	\end{align*}
 	

	\hfill \underline{(End of the estimation of $I_{221}$)}
		
	\bigskip
	
	\quad $\bullet$ \underline{Estimation of $I_{222}$}
	
	\bigskip
	
	Investigating
	\begin{equation*}
		I_{222(r,t)}=\bigg\vert \int_{A(R,3) \cap \set{y \geq 0}} \frac{y^\perp}{\abs{y}^2} \bigg( \Big( w_s(\Phi_t^{-1}(y-(0,r)))-w_s(y-(0,r)) \Big)- \Big(w_s(\Phi_t^{-1}(y+(0,r)))-w_s(y+(0,r)) \Big)  \bigg) \, dy \bigg\vert,
	\end{equation*}
	we can see that the integrand is same as $I_{221}$, but the integral range is more affirmative: It strictly avoids the origin where the integrand becomes singular. Therefore, it is far more easier to bound $I_{222}$ than $I_{221}$, using the fact that $w_s$ is compactly supported smooth function on $\overline{A(R,3)}$. Thus, we only provide the result to avoid redundancy: There exists a constant $M=M(R)$, where $R$ is the constant appearing in the bound of the term $I_{222}$, such that
	\begin{equation} \label{eq32}
		I_{222}(r,t) \leq Mtr.
	\end{equation}
	
	\hfill \underline{(End of the estimation of $I_{222}$)}
		
	\bigskip
	
	Combining two estimates for $I_{221}$ and $I_{222}$ leads to the bound of $I_{22}$: There exist $R=R(\alpha, \varepsilon)>0$ and $M=M(\alpha,\varepsilon)>0$ such that for all $0<r<R$ and $0 \leq t \leq \frac{1}{\left( \log \frac{1}{r}\right)^{1-\alpha+\varepsilon}}$,
	\begin{equation} \label{eq33}
		I_{22} \underset{\eqref{equ6} }{\leq} I_{221}+I_{222} \underset{\eqref{eq31}, \eqref{eq32} }{\leq} Mtr\log \log \frac{1}{r}.
	\end{equation}
	
	\hfill \underline{(End of the estimation of $I_{22}$)}
	
	\bigskip
	
	Finally, the Proposition \ref{estimation of I2} is deduced by combining \eqref{eq33} with the estimate for $I_{21}$ \eqref{equ4}.
\end{proof}

\bigskip

\underline{\textit{Proof of the Key Lemma $($Lemma \ref{key lemma}}$)$}.

Suppose $\alpha\in (0,1)$ and $\varepsilon, \delta>0$ are given. By Proposition \ref{estimation of I1}, Proposition \ref{estimation of I3}, and Proposition \ref{estimation of I2}, there exist $R=R(\alpha,\varepsilon)>0$ and $M=M(\alpha,\varepsilon)>0$ such that for all $0<r<R$ and $0 \leq t \leq \frac{1}{\left( \log \frac{1}{r}\right)^{1-\alpha+\varepsilon}}$,
\begin{equation*}
	\begin{aligned}
		\abs{u_g(\phi_r(t),t)-u_g(\phi_r(0),0)} \underset{\eqref{eq20}}&{\leq} I_1(r,t)+I_2(r,t)+I_3(r,t) \\
		\underset{\substack{\text{Prop }\ref{estimation of I1}, \text{ Prop }\ref{estimation of I3} \\ \text{Prop }\ref{estimation of I2} }}&{\leq} r \left( Mt+Mt \left( \log \frac{1}{r}\right)^{1-\alpha} e^{Mt\left(\log \frac{1}{r}\right)^{1-\alpha}}+ Mt \log \log \frac{1}{r} \right).
	\end{aligned}
\end{equation*}
Then, the elementary fact
\begin{equation*}
	\lim_{r \searrow 0} \frac{M}{\left( \log \frac{1}{r}\right)^{1-\alpha+\varepsilon}} =\lim_{r \searrow 0} \frac{M \left( \log \frac{1}{r}\right)^{1-\alpha}}{\left( \log \frac{1}{r}\right)^{1-\alpha+\varepsilon}} = \lim_{r \searrow 0} \frac{\log \log \frac{1}{r}}{\left( \log \frac{1}{r}\right)^{1-\alpha+\varepsilon}} =0
\end{equation*}
implies the Proposition \ref{key lemma}.

\hfill $\square$

\subsection{Proof of the Main Theorem: Breakdown Part}
\label{subsec: Proof of the Main Theorem: Breakdown Part}

To prove main theorem, Theorem \ref{main thm 2}, we point out one more elementary ingredient.

\begin{lemma} \label{phirt estimate lemma}
	Let $\alpha \in (0,1)$ be given. Then, for any $\varepsilon, \delta>0$, there exists $R=R(\alpha, \varepsilon)>0$ such that for all $0<r<R$ and $0 \leq t \leq \frac{1}{\left( \log \frac{1}{r}\right)^{1-\alpha+\varepsilon}}$,
	\begin{align*}
		\abs{\phi_r(t)-(0, r)} &< \delta , \\
		\abs{\frac{\phi_r (t)}{\abs{\phi_r(t)}}-(0,1)} &< \delta.
	\end{align*}
\end{lemma}
\begin{proof}
	This is just a corollary of \eqref{equ5}.
\end{proof}

\bigskip

\underline{\textit{Proof of Theorem \ref{main thm 2}}}.

It suffices to consider the case $\beta \in (1,2)$. Suppose $\beta \in (0,1)$ and $T_0>0$ are given. Define $\alpha, \varepsilon>0$ by
\begin{equation} \label{eq34}
	\alpha \coloneqq\frac{5-\beta}{4}, \qquad \varepsilon \coloneqq \frac{\beta-1}{4}, \qquad \delta \coloneqq \min \left( \frac{K}{4}, \frac{1}{10}\right),
\end{equation}
where $K>0$ is the constant in \eqref{Lip initial}. Then, for any $r$ with $0<r<R$, where $R>0$ is the constant in Proposition \ref{key lemma},
\begin{align*}
	&\sup_{0 \leq t \leq T_0} \nrm{g(\cdot,t)}_{\phi_\beta} \\
	&\leq  \sup_{0 \leq t \leq T_0} \sup_{0<\abs{x}<e^{-3}} \left( \log \frac{1}{\abs{x}}\right)^{\beta} \vert g(x,t)-\overbrace{g(0,t)}^{\overset{\eqref{g} }{=}0}\vert \\
	& \geq \Bigg( \log \frac{1}{\abs{\phi_r \left(\frac{1}{\left( \log \frac{1}{r}\right)^{1-\alpha+\varepsilon}}\right)}}\Bigg)^{\beta} \abs{g \left( \phi_r \left( \frac{1}{\left( \log \frac{1}{r}\right)^{1-\alpha+\varepsilon}} \right), \frac{1}{\left( \log \frac{1}{r}\right)^{1-\alpha+\varepsilon}}\right)} \\
		\underset{\eqref{equation for g integral} }&{=}\Bigg( \log \frac{1}{\abs{\phi_r \left(\frac{1}{\left( \log \frac{1}{r}\right)^{1-\alpha+\varepsilon}}\right)}}\Bigg)^{\beta} \bigg\vert \int_0^{\frac{1}{\left( \log \frac{1}{r}\right)^{1-\alpha+\varepsilon}}} \frac{u_g(\phi_r(s),s) \cdot \frac{\phi_r(s)}{\abs{\phi_r(s)}}}{\abs{\phi_r(s)}\log \frac{1}{\abs{\phi_r(s)}}} \, ds \bigg\vert \\
		&= \Bigg( \log \frac{1}{\underbrace{\abs{\phi_r \left(\frac{1}{\left( \log \frac{1}{r}\right)^{1-\alpha+\varepsilon}}\right)}}_{\underset{\text{Lem } \ref{phirt estimate lemma} }{\geq}\frac{9}{10}r}}\Bigg)^{\beta} \cdot \\
	& ~~ \bigg\vert \int_0^{\frac{1}{\left( \log \frac{1}{r}\right)^{1-\alpha+\varepsilon}}}  \frac{\overbrace{(u_g(\phi_r(s),s) - u_g(\phi_r(0),0))\cdot \frac{\phi_r(s)}{\abs{\phi_r(s)}}}^{\underset{\text{Lem } \ref{key lemma} }{\leq} \frac{K}{4}r}+\overbrace{u_g(\phi_r(0),0) \cdot (0,1)}^{\underset{\eqref{Lip initial} }{\leq} -Kr} + \overbrace{u_g(\phi_r(0),0) \cdot \left( \frac{\phi_r(s)}{\abs{\phi_r(s)}}-(0,1)\right)}^{\underset{\text{Lem } \ref{phirt estimate lemma} }{\leq}\frac{K}{4}r}}{\underbrace{\abs{\phi_r(s)}\log \frac{1}{\abs{\phi_r(s)}}}_{\underset{\text{Lem } \ref{phirt estimate lemma}}{\geq}\frac{9}{10}r \log \frac{1}{r}}}    \, ds \bigg\vert \\
	& \geq \left( \frac{10}{9}\right)^{1+\beta} \frac{K}{2} \left( \log \frac{1}{r}\right)^{\beta-2+\alpha-\varepsilon} \\
	\underset{\eqref{eq34} }&{=} \left( \frac{10}{9}\right)^{1+\beta} \frac{K}{2} \left( \log \frac{1}{r}\right)^{\frac{\beta-1}{2}} \xrightarrow[\quad  r \searrow 0 \quad ]{} \boldsymbol{\infty}.
\end{align*}

\hfill $\square$

\end{document}